\documentclass[twoside,12pt]{amsart}

\usepackage{amsmath,amsthm,amsfonts,amssymb,verbatim,enumerate}
\usepackage{array,charter}
\usepackage{longtable}
\usepackage{verbatim}
\usepackage{pdfpages}

\usepackage{latexsym,mathptm, times,mathcomp,lscape,textcomp}
   \usepackage{tikz,stmaryrd}
\usepackage{calligra}
\usepackage{calrsfs}
\usepackage[mathscr]{euscript}

\input amssym.def
\input amssym 
\input xypic
\input xy
\xyoption{all}

\usepackage[right=1.5in,left=1.5in,top=1in,bottom=1in]{geometry}
\usepackage{amssymb}

\def \ms{\medskip}
\def \bs{\bigskip}
\def\boxx{\theta}
\def\ts{\textstyle}
\def\Sym{\operatorname{Sym}}
\def\kk{{\pmb k}}
\def\im{\operatorname{im}}
\def\w{\wedge}
\def\ann{\operatorname{ann}}
\def\p{\oplus}
\def\Hom{\operatorname{Hom}}
\def\mult{\operatorname{mult}}
\def\t{\otimes}

\def\a{\alpha}

\def\bw{\bigwedge}

\def\rank{\operatorname{rank}}
\def\e{\epsilon}

\newtheorem{theorem}{Theorem}[section]

\newtheorem{lemma}[theorem]{Lemma}
\newtheorem{observation}[theorem]{Observation}
\newtheorem{proposition}[theorem]{Proposition}

\newtheorem{claim-no-advance}[equation]{Claim}

\theoremstyle{definition}
\newtheorem{conventions}[theorem]{Conventions}
\newtheorem{data}[theorem]{Data}

\newtheorem{notation}[theorem]{Notation}
\newtheorem{definition}[theorem]{Definition}
\newtheorem{Cases}[theorem]{Cases}
\newtheorem{chunk}[theorem]{}
\newtheorem{chunk-no-advance}[equation]{}
\newtheorem{example}[theorem]{Example}

\newtheorem{remark}[theorem]{Remark}

\newtheorem*{Remark}{Remark}

\numberwithin{equation}{theorem}

\begin{document}

\baselineskip=16pt

\title[Weak Lefschetz property and socle degree three]{The weak Lefschetz property for  Standard graded, Artinian Gorenstein  algebras of embedding dimension four and  socle degree three}

\date{\today}

\author[Andrew R. Kustin]
{Andrew R. Kustin}

\address{Department of Mathematics\\ University of South Carolina\\\newline
Columbia SC 29208\\ U.S.A.} \email{kustin\@math.sc.edu}



\thanks{We made extensive use of   Macaulay2 \cite{M2}. We are very appreciative of this computer algebra system.}

\begin{abstract} Let $\kk$ be an arbitrary field and  $A$ be  a standard graded
Artinian Gorenstein  
$\kk$-algebra of embedding dimension four and  socle degree three.
Then, except for exactly one exception, $A$ has the weak Lefschetz property. Furthermore, the exception occurs only in characteristic two.
 \end{abstract}

\maketitle

\tableofcontents

\section{Introduction.}

The Lefschetz property is a ring-theoretic abstraction of the Hard Lefschetz
Theorem for compact K\"ahler manifolds.
Let $\kk$ be a field. A graded $\kk$-algebra $A$, equal to $\bigoplus_{i=0}^sA_i$, has the weak Lefschetz property if there is a linear form $\ell\in A_1$ so that  multiplication by $\ell$ from $A_i$ to $A_{i+1}$ has maximal rank for each $i$.
(Similarly, $A$
has the strong Lefschetz property  if multiplication by $\ell^s$ has maximal
rank in each degree for every positive integer $s$ for some $\ell\in A_1$.) Stanley introduced the concept in \cite{S80} where he proved that if the characteristic of $\kk$ is zero, $P=\kk[x_1,\dots,x_n]$, and $A=P/(x_1^{a_1},\dots,x_n^{a_n})$, then $A$ has the strong Lefschetz property. Other early proofs of Stanley's theorem were also given by \cite{RRR91} and \cite{W87}.

Our interest in the weak Lefschetz property for fields of arbitrary characteristic is a consequence of our desire to record 
 the minimal resolution of $A$ by free $P$-modules in terms of the coefficients of the Macaulay inverse system which determines $A$. See \cite{EKK1}, \cite{EKK2}, and especially \cite{EKK3}, 
where this project is carried out for 
 quotient rings $A$ with Castelnuovo regularity two as a $P$-module, and  \cite{qp} for quotient rings $A$ of 
 regularity three 
 in three variables.
(We also have results along these lines pertaining to ideals of 
regularity three in four and five variables.) 
We are able to carry out this project  
provided $A$ has the weak Lefschetz property, independent of the characteristic of $\kk$.
The absence of the weak Lefschetz property is an obstruction to  this project; but positive characteristic, in and of itself, does not  cause any problem.

\bs The weak Lefschetz property is very sensitive to change in characteristic. 
In positive characteristic $p$, 
Brenner and 
 Kaid \cite{BK11} 
 gave an explicit description of those $d$ and $p$ for which  
$P/(x^d_1 , x^d_2 ,\dots, x^d_n)$ has the weak Lefschetz property 
 when $P$ equals $\kk[x_1,\dots,x_n]$, $\kk$ is a field of characteristic $p$, and $n=3$.
The paper \cite{KV14}  is devoted to the analogous project for $4\le n$. 

\bs Artinian rings with socle degree three are somewhat mysterious. 
 B{\o}gvad's \cite{B} examples of Artinian Gorenstein rings with transcendental Poincar\'e series have socle degree three. Rossi and \c Sega \cite{RS14} prove that if $R$ is a compressed Artinian Gorenstein local ring with 
socle degree not equal to three, then the Poincar\'e series of all finitely generated $R$-modules  are rational, sharing a common
denominator. Similarly, it is shown in \cite{KSV18} that if $R$ is a compressed local Artinian ring
with odd top socle degree  at least five, then the Poincar\'e
series of all finitely generated $R$-modules are rational, sharing a common denominator. 
The same conclusion does not hold when the top socle degree is three.

\bs 

Theorem~\ref{main} is the main result in the paper.

\begin{theorem}
\label{main}Let $\kk$ be a field 
 and $A$ be a standard graded Artinian Gorenstein $\kk$-algebra of embedding dimension four and socle degree three.
If the characteristic of $\kk$ is different than two, then $A$ has the weak Lefschetz property. 
If the characteristic of $A$ is equal to two, then  
$A$ has the weak Lefschetz property
if and only if $A$ is not isomorphic to
\begin{equation}\label{1.1.1}\frac{\kk[x,y,z,w]}{(xy,xz,xw,y^2,z^2,w^2,x^3+yzw)}.\end{equation}
\end{theorem}

\bs We first demonstrate that the $\kk$-algebra $A$ of (\ref{1.1.1}) does not satisfy the weak Lefschetz property. Indeed, let $\ell=ax+by+cz+dw$ be an arbitrary nonzero linear form in $A$. If at least one of the parameters $b$, $c$, or $d$ is nonzero, then let $\ell'$ be the nonzero linear form $\ell'=by+cz+dw\in A_1$. Observe that $$\ell \ell'=ax (by+cz+dw)+b^2y^2+c^2z^2+d^2w^2=0\in A.$$ 
 If $b=c=d=0$, then let $\ell'$ be the nonzero linear form $\ell'=y$ of $A_1$.
Observe that 
$\ell\ell'=axy=0$ in $A$.
In either case, the arbitrary nonzero linear form $\ell$ is zero divisor on $A$.

\bs In the rest of the paper we show that if $A$ satisfies the hypotheses of Theorem~\ref{main}, but is not  isomorphic to the ring of (\ref{1.1.1}), then $A$ has the weak Lefschetz property. We use a Macaulay inverse system for $A$.

Let $U$ be the vector space $A_1$, $P$ be the polynomial ring $P=\Sym_\bullet U$, $I$ be a homogeneous ideal of $P$ with $A$ isomorphic to $P/I$, $U^*$ be the dual space $\Hom_{\kk}(U,\kk)$ of $U$,  and $D_\bullet U^*$ be the divided power $\kk$-algebra $\bigoplus_{0\le i} D_iU^*$, with $$D_iU^*=\Hom_{\kk}(\Sym_i U,\kk).$$
The rules for a divided power algebra are recorded in \cite[Section~7]{GL} or \cite[Appendix 2]{Eis}. (In practice these rules say that $w^{(n)}$ behaves like $w^n/(n!)$ would behave if $n!$ were a unit in $R$.)

Macaulay duality guarantees that 
$$\ann_{D_\bullet U^*}I$$ is a cyclic $P$-submodule of $D_\bullet U^*$ generated by an element 
in $D_3U^*$. (Any generator 
$\ann_{D_\bullet U^*}I$
in $D_3U^*$ 
is called 
a Macaulay inverse system for $A$.) Furthermore, if $\phi_3$ is a Macaulay inverse system for $A$, then
$$I=\ann_P \phi_3.$$The hypothesis that $A$ has embedding dimension four ensures that
\begin{equation}\label{non-trivial}\ell \phi_3\neq 0\quad\text{for any nonzero $\ell$ in $U$.}\end{equation}
We observe that the Macaulay inverse system for the $\kk$-algebra $A$ of (\ref{1.1.1}) is \begin{equation}\label{evil}\phi_3=x^{*(3)}+y^*z^*w^*.\end{equation}

\bs In section~\ref{convenient} we put the Macaulay inverse system into the form
\begin{align}\a\phi_3={}&x^{*(3)}+x^*\phi_{2,0}+\phi_{3,0} \label{form1}
\intertext{or}
\a\phi_3={}&x^{*(2)}y^*+x^*\phi_{2,0}+\phi_{3,0}, \label{form2}
\end{align}
for some unit $\a$,
with $\phi_{i,0}\in D_i(\kk y^*\p \kk z^*\p\kk w^*)$.
We study various cases depending on whether $\phi_3$ has form (\ref{form1}) or (\ref{form2}) and also depending on how complicated $\phi_{2,0}$ is. We treat the four significant cases in Sections \ref{cubic,r=2}, \ref{cubic,r=1}, \ref{nocubic,and}, and \ref{nocubic,other}.
 All of our calculations employ the homomorphism $$\ts \Gamma_{\phi_3}:D_dU\t \bw^d U\to \bw^dU^*$$ which is introduced in section \ref{Gamma}. The connection between $\Gamma_{\phi_3}$ and the weak Lefschetz property is explained in Lemma~\ref{28.10}. Roughly speaking, 
$$\Gamma_{\phi_3}(\ell^{(d)}\t x_1\w\ldots\w x_d)=\ell x_1\phi_3\w \ldots\w \ell x_d\phi_3\in \ts\bw^dU^*,$$ 
for $\ell,x_1,\ldots,x_d\in U$ and $\phi_3\in U^*$. It is shown in Lemma~\ref{28.10} that if $\phi_3$ is a Macaulay inverse system for $A$ and $\ell\phi_3$ is nonzero for all nonzero $\ell$ in $U$, then
\begin{equation}\notag \text{$A$ has the weak Lefschetz property} \iff \Gamma_{\phi_3} \text{ is not identically zero.}\end{equation} The hypothesis that $\ell\phi_3$ is nonzero whenever $\ell$ in $U$ is nonzero is innocuous. It merely says that the embedding dimension of $A$ is equal to the number of variables of $P$. If this hypothesis is not satisfied, then one can view $A$ as a quotient of a polynomial ring with one fewer variable than $P$ has.

\ms 
Traditionally, the 
Lefschetz properties are studied in a graded $\kk$-algebra where 
$\kk$ is a  field 
 of characteristic zero.
In particular, for example,  Gondim and Zappala \cite[Cor. 5.5]{GZ19} have proven  that a standard graded Gorenstein $\kk$-algebra of small codimension, with socle degree three, and presented by quadrics, has the weak Lefschetz property, provided the field $\kk$ has characteristic zero. 
 Duality is obtained using the algebra of differential operators  $Q = \kk[\frac{\partial }{\partial x_1},\dots,\frac{\partial }{\partial x_n}]$.
Every graded Artinian Gorenstein algebra $A$ has a 
 presentation of the form 
 \begin{equation}\label{A_F}A\cong Q/\ann_QF,\end{equation}
for some $n$ and some homogeneous polynomial $F\in \kk[x_1, . . . , x_n]$, where $$\ann_QF=\{\phi\in Q 
\mid \phi(F(x)) = 0\}.$$
The Hessian 
of the form $F$ is the determinant of the  square matrix $(\frac{\partial^2 F}{\partial x_i\partial x_j})$ of  second order partial derivatives of $F$.
 When $\kk$ has characteristic zero, Watanabe (see \cite{W00,MW09}) has shown that 
(\ref{A_F}) fails to have  the Strong Lefschetz Property if and only if one of the
non-trivial higher Hessians of $F$ vanishes. 
This result has been generalized to the weak Lefschetz property using ``mixed
Hessians'', see \cite{GZ19}.

The study of when the Hessian of a homogeneous form vanishes has a very long history and continues to be active. For the time being keep $\kk$ a field of characteristic zero and let $F$ be a homogeneous form in $\kk[x_1,\dots,x_n]$.
Hesse \cite{H51,H59} believed that the Hessian of 
$F$ vanishes if and only if the projective variety $X$, defined by $F$, 
is a cone.
 However, Gordan and Noether \cite{GN76}
proved that while Hesse's claim is true when the degree of $F$ is $2$ or $n \le 4$, it is false for
$5\le n$ and for forms of degree at least three. 
Gordan and Noether realized that $X$ being a cone
is equivalent to the condition that the partial derivatives of $F$ are $\kk$-linearly dependent,
while $F$ has vanishing Hessian if and only if the partial derivatives of $F$
are $\kk$-algebraically dependent. 
A class of forms $F$, discussed both in \cite{GN76}
and by Perazzo \cite{Per00},  
with vanishing Hessian and for which $V (F)$
is not a cone, are the Perazzo forms
$$F =  x_1p_1 + \dots + x_ap_a + G \in \kk[x_1, \dots ,x_a, u_1,\dots, u_b],$$
where  $p_1,\dots, p_a$ and $G$ are in  $\kk[u_1, \dots , u_b]$, and $p_1, \dots, p_a$ are
linearly independent, but algebraically dependent. In particular, for
$ a + b = 5$, all non-cones defined by a form with vanishing Hessian
are defined by a
Perazzo form. 
(See
\cite{GN76} or \cite[Theorem 7.3]{WdeB20}.) 
The Lefschetz properties of rings defined by Perazzo forms in characteristic zero are investigated in the recent papers \cite{FMM23,AA}.

J. Watanabe and M.~de Bondt \cite{WdeB20} have written a detailed modern argument for the Gordan-Noether Theorem. The paper \cite{BFP22} uses geometric 
techniques to give a new proof of the Gordan-Noether Theorem. In both of these papers the field $\kk$ has characteristic zero. 

\ms 
We work in arbitrary characteristic; so we do not  take literal partial derivatives. Instead we use the divided power algebra $D_\bullet U^*$ which is associated to the polynomial ring $P=\Sym_\bullet U$ for the vector space $U$ over the field $\kk$. We replace the ``Hessian of $F$'' with the homomorphism ``$\Gamma_{\phi_3}$''  of Section~\ref{Gamma}.

\bs
In Section~\ref{3var} we state and prove the three variable version of the Main Theorem (Theorem \ref{main}). Our precise formulation of the three variable version (see Lemma~\ref{lemma3}) is used in the inductive part of the proof of Theorem~\ref{main}. (See the case $r=0$ in Lemma~\ref{case3a}.)
Furthermore, we prove the three variable version using the same argument as we use for the four variable version; except there are fewer cases and each calculation is more straightforward. The reader might want to read Section~\ref{3var} as a preparation for reading the proof of Theorem~\ref{main}.

\section
{Notation, conventions, and elementary results.}\label{prelims}

\subsection{The language.}
\begin{conventions}
\label{The language}
\begin{enumerate}[\rm(a)] 
 \item The graded algebra $A=\bigoplus_{0\le i}A_i$ is a {\it standard graded $A_0$-algebra} if $A_1$ is finitely generated as an $A_0$-module  and 
$A$ is generated as an $A_0$-algebra by $A_1$.

\item Let $\kk$ be a field and $A=\bigoplus A_i$ be a standard graded $\kk$-algebra. Then $A$ has the {\it weak Lefschetz property}  if there exists a linear form $\ell$ of $A_1$ such that the $\kk$-module homomorphism
$$\mu_{\ell}:A_i\to A_{i+1}$$ has maximal rank for each index $i$, where $\mu_{\ell}$ is multiplication by $\ell$. 
(A homomorphism $\xi:V\to W$ of finitely generated $\kk$-modules has {\it maximal rank} if $\rank \xi$ is equal to $\min\{\dim V,\dim W\}$.)

\item If $A=\bigoplus_{i=0}^\sigma A_i$ 
 is an Artinian standard-graded $\kk$-algebra, then $A$ is {\it Gorenstein with socle degree $\sigma$} if $A_\sigma$ is a one dimensional vector space and every ideal of $A$ contains $A_\sigma$.

\item
In this paper $\kk$ is an arbitrary field (unless otherwise noted) and $\Hom$, $\Sym$, $D$, $\bw$, $\w$, and $\t$ mean  $\Hom_\kk$, $\Sym^\kk$, $D^\kk$, $\bw_\kk$, $\w_\kk$ and $\t_\kk$, respectively. 
\item If $U$ is a vector space over the field $\kk$, then $T_\bullet U$, $\Sym_\bullet U$, $D_\bullet U$, and $\bw^\bullet U$ are the tensor algebra, symmetric algebra, divided power algebra, and exterior algebra of $U$ over $\kk$, respectively. See, for example, \cite{Nor} or \cite{Eis}.

\item  If $M$ is a matrix, then $\det M=|M|$  is the determinant 
of $M$. \item If $f$ is a homomorphism, then we write $\im f$ and $\ker f$ for the image and kernel of $f$, respectively.
\item If $U$ is a vector space, then $\dim U$ is the dimension of $U$ as a vector space.
\end{enumerate}
\end{conventions}
\begin{conventions}
Let $U$ be a finite dimensional vector space  over the field $\kk$.
\begin{enumerate}[\rm(a)]
\item 
 Let $U^*$ represent $\Hom_\kk(U,\kk)$ and $D_\bullet U^*$ represent the divided power algebra $\bigoplus_{i=0}^\infty D_iU^*$ for $$D_iU^*=\Hom_\kk(\Sym_iU,\kk).$$

\item\label{co-mu} The co-multiplication map $\Delta: D_iU\to T_iU$ is the composition
\begin{align*}D_iU&{}\subseteq  D_\bullet U\xrightarrow{D_{\bullet}(d)}D_{\bullet}(\underbrace{U\p\dots\p U}_{i\text{ times}})=\underbrace{D_{\bullet}U\t \dots \t D_{\bullet}U}_{i\text{ times}}\\{}&\xrightarrow{\text{projection}}\underbrace{D_{1}U\t \dots \t D_{1}U}_{i\text{ times}},\end{align*} where $D_{\bullet}(d)$ is the divided power algebra homomorphism induced by the diagonal map
$$d:U\to \underbrace{U\p\dots\p U}_{i\text{ times}}, \quad \text{given by $d(u)=(\underbrace{u,\dots,u}_{i\text{ times}})$, for $u\in U$}.$$ 
In particular, $\Delta:D_3U\to T_3U$ sends $x^{(2)}y$ to $x\t x\t y+x\t y\t x+ y\t x\t x$, for $x$ and $y$ in $U$.
\item If $x_1,\dots,x_d$ is a basis for $U$, then the set of monomials of degree $i$ in  
$x_1,\dots,x_d$, denoted $\binom{x_1,\dots,x_d}{i}$, is a basis for the $i$-th symmetric power, $\Sym_iU$, of $U$ and the set of homomorphisms
$$\left\{ m^*\left| m\in \binom{x_1,\dots,x_d}{i}\right.\right\}$$ is a basis for $D_iU^*$ where $m^*:\Sym_i U\to \kk$ is the $\kk$-module homomorphism with
\begin{equation}\label{*}m^*(m')=\begin{cases} 1,&\text{if $m=m'$, and }\\0,&\text{if $m\neq m'$,}\end{cases}\end{equation} for $m,m'$ in $\binom{x_1,\dots,x_d}{i}$. \item We make 
much use of the structure of $D_\bullet U^*$ as a module over $\Sym_\bullet U$.
If $v_i\in D_iU^*$ and $u_j\in \Sym_jU$, then $u_jv_i$ is the element of $D_{i-j}U^*$ which sends $u_{i-j}$ in $\Sym_{i-j}U$ to $v_i(u_ju_{i-j})$.
In particular, if $m$ and $m'$ are monomials in $\Sym_\bullet U$
(with respect to some basis $x_1,\dots,x_d$ for $U$ and $*$ is defined as in (\ref{*})), then
$$m'(m)^*= \begin{cases} (\frac m{m'})^*,&\text{if $m'$ divides $m$},\\
0,&\text{if $m'$ does not divide $m$}.\end{cases}
$$
\item If $u^*$ is an element of $U^*$, then we write $u^{*(n)}$ for the element $(u^*)^{(n)}$ in $D_nU^*$.

\item If $x_1,\dots,x_d$ is a basis for the vector space $U$, then $x_1^*,\dots,x_d^*$ is the dual basis for $U^*$. Similarly, if then $x_1^*,\dots,x_d^*$ is a basis for a vector space $U^*$, then $x_1,\dots,x_d$ is the dual basis for the vector space $U$.
\item  If $u_i\in \Sym_iU$ and $\phi_i\in D_iU^*$, then $u_i\phi_i=\phi_iu_i$ is an element of $\kk$.
\end{enumerate}
 \end{conventions}

The following data is used throughout the paper.
\begin{data}
\label{28.4}
 Let $\kk$ be a field,  $U$ be a $d$-dimensional vector space over $\kk$, $P$ be the polynomial ring $P=\Sym_\bullet U$, $\phi_3$ be a non-zero element of $D_3U^*$, $I=\ann_P(\phi_3)$, and $A_{\phi_3}$ be the standard graded Artinian Gorenstein $\kk$-algebra
$A_{\phi_3}={P}/{I}$. The socle degree of $A$ is three. 
\end{data}

\subsection{The homomorphisms $\bowtie$, $p_{\phi_3}$,  and $\Gamma_{\phi_3}$.}
\label{Gamma}

\begin{notation} 
Let $\kk$ be a field, $E$ and $G$ be $\kk$-modules, and $m$ be a positive integer. Each pair  of elements $( X,Y )$, with $X\in  
D_ m E$ and $Y \in  \bw^m G$, gives rise to an element of $\bw^m( E\t G  )$, which we denote
by $X \bowtie Y$. We now give the definition of $X \bowtie Y$. Consider the composition
$$D_m E\t 
T_m G \xrightarrow{\Delta\t 1}
T_m E \t T_m G\xrightarrow{\xi}  \ts\bw^m ( E \t G ), $$
where $\Delta:D_m E\to T_mE$ is co-multiplication (see Convention~\ref{The language}.(\ref{co-mu})) and 
$$\xi\big(( x_1 \t \ldots \t x_ m )\t ( y_ 1 \t\ldots \t y_m )\big) = ( x_1\t y_1 ) \w\ldots \w ( x_m\t y_m ),$$ for
$x_i \in E$ and
$y_i \in  G$. It is easy to see that the above composition factors  through
$D_mE\t \bw ^m G$. 
Let $X\t Y\mapsto  X \bowtie Y$ be the resulting map from $D_m E\t \bw^m G$ to $\bw^m(E\t G)$. This map is used 
in \cite{K08} 
 and is called $\langle -,-\rangle$ in \cite[III.2]{ABW}.
\end{notation}

\begin{definition}
\label{28.7} Adopt  
Data~\ref{28.4}.
\begin{enumerate}[\rm(a)]

\item \label{28.7.a} Define the $\kk$-module homomorphism $p_{\phi_3}:\Sym_2U\to U^*$ by $p_{\phi_3}(u_2)=u_2(\phi_3)$, for $u_2\in \Sym_2U$. 

\item\label{28.7.b} Define the $\kk$-module homomorphism $\ts\Gamma_{\phi_3}:D_dU\t \ts\bw^dU\to \bw^dU^*$ 
to be the composition
$$\ts D_dU\t\ts\bw^dU\xrightarrow{\bowtie}\bw^d(U\t U)\xrightarrow{\bw^d\mult}
\bw^d(\Sym_2U)\xrightarrow{\bw^dp_{\phi_3}}\bw^dU^*,$$where $$\mult:U\t U\to \Sym_2U$$ is multiplication in the Symmetric algebra $\Sym_\bullet U$.
\end{enumerate}
\end{definition}

\begin{example}
\label{28.8} Adopt Data~\ref{28.4}. Let $\Gamma_{\phi_3}$ be the $\kk$-module homomorphism of Definition~\ref{28.7}.(\ref{28.7.b}). Let $\ell_1,\dots,\ell_d$ be a basis for $U$.
 
If $d=3$, then \begingroup\allowdisplaybreaks
\begin{align*}
\Gamma_{\phi_3}(\ell_1^{(3)}\t \ell_1\w \ell_2\w \ell_3)={}&\phantom{++}\ell_1^2\phi_3\w \ell_1 \ell_2\phi_3\w \ell_1 \ell_3\phi_3,\\
\Gamma_{\phi_3}(\ell_1^{(2)}\ell_2\t \ell_1\w \ell_2\w \ell_3)={}
&\begin{cases}
\phantom{+}\ell_1^2\phi_3\w \ell_2^2\phi_3\w \ell_1 \ell_3\phi_3\\
+\ell_1^2\phi_3\w \ell_1 \ell_2\phi_3\w \ell_2 \ell_3\phi_3,
\text{ and}
\end{cases}\\
\Gamma_{\phi_3}(\ell_1\ell_2\ell_3\t \ell_1\w \ell_2\w \ell_3)={}&
\begin{cases}
\phantom{+}\ell_1^2\phi_3\w \ell_2^2\phi_3\w \ell_3^2\phi_3\\
+2 \ell_1\ell_2\phi_3\w \ell_2\ell_3\phi_3\w \ell_1\ell_3\phi_3.\end{cases}
\end{align*}\endgroup

If $\dim U=4$, then 
\begingroup\allowdisplaybreaks
\begin{align*}
\Gamma_{\phi_3}(\ell_1^{(4)}\t \ell_1\w \ell_2\w \ell_3 \w \ell_4)={}&\phantom{++}
\ell_1^2\phi_3\w \ell_1\ell_2\phi_3\w \ell_1\ell_3\phi_3\w \ell_1\ell_4\phi_3,\\
\Gamma_{\phi_3}(\ell_1^{(3)}\ell_2\t \ell_1\w \ell_2\w \ell_3 \w \ell_4)={}&
\begin{cases}
\phantom{+}\ell_1^2\phi_3\w \ell_2^2\phi_3\w \ell_1\ell_3\phi_3\w \ell_1\ell_4\phi_3\\
+\ell_1^2\phi_3\w \ell_1\ell_2\phi_3\w \ell_2\ell_3\phi_3\w \ell_1\ell_4\phi_3\\
+\ell_1^2\phi_3\w \ell_1\ell_2\phi_3\w \ell_1\ell_3\phi_3\w \ell_2\ell_4\phi_3,\end{cases}
\\
\Gamma_{\phi_3}(\ell_1^{(2)}\ell_2^{(2)}\t \ell_1\w \ell_2\w \ell_3 \w \ell_4)={}& 
\begin{cases}
\phantom{+}\ell_1^2\phi_3\w \ell_1\ell_2\phi_3\w \ell_2\ell_3\phi_3\w \ell_2\ell_4\phi_3\\
+\ell_1^2\phi_3\w \ell_2^2\phi_3\w \ell_1\ell_3\phi_3\w \ell_2\ell_4\phi_3\\
+\ell_1^2\phi_3\w \ell_2^2\phi_3\w \ell_2\ell_3\phi_3\w \ell_1\ell_4\phi_3\\
+\ell_1\ell_2\phi_3\w \ell_2^2\phi_3\w \ell_1\ell_3\phi_3\w \ell_1\ell_4\phi_3,\\
\end{cases}\\
\Gamma_{\phi_3}(\ell_1^{(2)}\ell_2\ell_3\t \ell_1\w \ell_2\w \ell_3 \w \ell_4)={}&
\begin{cases}
\phantom{+}\ell_1^2\phi_3\w \ell_1\ell_2\phi_3\w \ell_2\ell_3\phi_3\w \ell_3\ell_4\phi_3&\\
+\ell_1^2\phi_3\w \ell_1\ell_2\phi_3\w \ell_3^2\phi_3\w \ell_2\ell_4\phi_3&\\
+\ell_1^2\phi_3\w \ell_2^2\phi_3\w \ell_1\ell_3\phi_3\w \ell_3\ell_4\phi_3&\\
+\ell_1^2\phi_3\w \ell_2\ell_3\phi_3\w \ell_1\ell_3\phi_3\w \ell_2\ell_4\phi_3&\\
+\ell_1^2\phi_3\w \ell_2^2\phi_3\w \ell_3^2\phi_3\w \ell_1\ell_4\phi_3&\\
+2\ell_1\ell_2\phi_3\w \ell_2\ell_3\phi_3\w \ell_1\ell_3\phi_3\w \ell_1\ell_4\phi_3, \text{ and}&\\
\end{cases}\\
\Gamma_{\phi_3}(\ell_1\ell_2\ell_3\ell_4\t \ell_1\w \ell_2\w \ell_3 \w \ell_4)
={}&\begin{cases}
\phantom{+}\ell_1^2\phi_3\w \ell_2^2\phi_3\w \ell_3^2\phi_3\w \ell_4^2\phi_3\\
+2\ell_1^2\phi_3\w \ell_2\ell_3\phi_3\w \ell_3\ell_4\phi_3\w \ell_2\ell_4\phi_3\\
+2\ell_1\ell_3\phi_3\w \ell_2^2\phi_3\w \ell_3\ell_4\phi_3\w \ell_1\ell_4\phi_3\\
+2\ell_1\ell_2\phi_3\w \ell_2\ell_4\phi_3\w \ell_3^2\phi_3\w \ell_1\ell_4\phi_3\\
+2\ell_1\ell_2\phi_3\w \ell_2\ell_3\phi_3\w \ell_1\ell_3\phi_3\w \ell_4^2\phi_3.\\
\end{cases}
\end{align*}
\endgroup
For example, we calculate $\Gamma_{\phi_3}(\ell_1^{(2)}\ell_2\ell_3\t \ell_1\w\ell_2\w\ell_3\w \ell_4)$ explicitly. Observe that
$$\Delta(\ell_1^{(2)}\ell_2\ell_3)=\begin{cases}
\phantom{+}
 \ell_1\t \ell_1\t\ell_2\t \ell_3
+\ell_1\t \ell_2\t\ell_1\t \ell_3
+\ell_2\t \ell_1\t\ell_1\t \ell_3\\
+\ell_1\t \ell_1\t\ell_3\t \ell_2
+\ell_1\t \ell_2\t\ell_3\t \ell_1
+\ell_2\t \ell_1\t\ell_3\t \ell_1\\
+\ell_1\t \ell_3\t\ell_1\t \ell_2
+\ell_1\t \ell_3\t\ell_2\t \ell_1
+\ell_2\t \ell_3\t\ell_1\t \ell_1\\
+\ell_3\t \ell_1\t\ell_1\t \ell_2
+\ell_3\t \ell_1\t\ell_2\t \ell_1
+\ell_3\t \ell_2\t\ell_1\t \ell_1.
\end{cases}$$ It follows that 
\begingroup\allowdisplaybreaks
\begin{align*}{}&\Gamma_{\phi_3}(\ell_1^{(2)}\ell_2\ell_3\t \ell_1\w\ell_2\w\ell_3\w \ell_4)\\
{}={}&\begin{cases}
\phantom{+}
 \ell_1^2\phi_3\w \ell_1\ell_2\phi_3\w\ell_2\ell_3\phi_3\w \ell_3\ell_4\phi_3
+\ell_1^2\phi_3\w \ell_2^2\phi_3\w\ell_1\ell_3\phi_3\w \ell_3\ell_4\phi_3\\
+\underline{\ell_2\ell_1\phi_3\w \ell_1\ell_2\phi_3\w\ell_1\ell_3\phi_3\w \ell_3\ell_4\phi_3}
+\ell_1^2\phi_3\w \ell_1\ell_2\phi_3\w\ell_3^2\phi_3\w \ell_2\ell_4\phi_3\\
+\ell_1^2\phi_3\w \ell_2^2\phi_3\w\ell_3^2\phi_3\w \ell_1\ell_4\phi_3
+\underline{\ell_2\ell_1\phi_3\w \ell_1\ell_2\phi_3\w\ell_3^2\phi_3\w \ell_1\ell_4\phi_3}\\
+\ell_1^2\phi_3\w \ell_3\ell_2\phi_3\w\ell_1\ell_3\phi_3\w \ell_2\ell_4\phi_3
+\underline{\ell_1^2\phi_3\w \ell_3\ell_2\phi_3\w\ell_2\ell_3\phi_3\w \ell_1\ell_4\phi_3}\\
+\ell_2\ell_1\phi_3\w \ell_3\ell_2\phi_3\w\ell_1\ell_3\phi_3\w \ell_1\ell_4\phi_3
+\underline{\ell_3\ell_1\phi_3\w \ell_1\ell_2\phi_3\w\ell_1\ell_3\phi_3\w \ell_2\ell_4\phi_3}\\
+\ell_3\ell_1\phi_3\w \ell_1\ell_2\phi_3\w\ell_2\ell_3\phi_3\w \ell_1\ell_4\phi_3
+\underline{\ell_3\ell_1\phi_3\w \ell_2^2\phi_3\w\ell_1\ell_3\phi_3\w \ell_1\ell_4\phi_3}.
\end{cases}\end{align*}\endgroup  
The underlined summands are zero and the last two summands which are not underlined   are equal. 
\end{example}

\ms
Remark~\ref{ff} is used in the proof of Observation~\ref{28.9}.
\begin{remark} 
Field extensions are always faithfully flat.  
Let $V$ be a vector space over a field $\kk$, $K$ be a field extension of $\kk$, and $v\in V$. If $v\t_{\kk}1$ is zero in $V\t_\kk K$, then $v$ is zero in $V$.\label{ff}\end{remark}

\begin{observation} 
Adopt Data~{\rm\ref{28.4}}.
\label{28.9} If the homomorphism $$\ts \Gamma_{\phi_3}:D_dU\t \bw^dU\to \bw^dU^*$$ of Definition~{\rm\ref{28.7}} satisfies $\Gamma_{\phi_3}(\ell^{(d)}\t \omega_U)=0$ for all $\ell\in U$ and some basis element  $\omega_U$ of $\bw^dU$, then $\Gamma_{\phi_3}$ is identically zero.
\end{observation}

\begin{proof}
Fix a basis $\omega_U$ for $\bw^dU$. 
The map $$\ts \Gamma_{\phi_3}(-\t \omega_U):U\to \bw^dU^*$$ is a linear transformation of vector spaces over a field. 
In light of Remark~\ref{ff}, 
 it suffices to prove the assertion when $\kk$ has a large number of elements. 

The proof of Observation~\ref{28.9} is obtained by iterating the following claim.

\begin{claim-no-advance}
\label{28.9.1} Let $X$ be an element of $D_\delta U$ for some integer $\delta$ with $0\le \delta\le d$. If $\Gamma_{\phi_3}(\ell^{(d-\delta)}X\t \omega_U)=0$ for all $\ell\in U$, then $\Gamma_{\phi_3}(\ell_1^{(e_1)}\ell_2^{(e_2)}X\t \omega_U)=0$ for all $\ell_1,\ell_2$ in $U$ and all nonnegative integers $e_1$ and $e_2$ with $e_1+e_2=d-\delta$.
\end{claim-no-advance}

\ms \noindent{\it Proof of Claim~{\rm\ref{28.9.1}}.}
If $\ell_1$ and $\ell_2$ are in $U$ and $a_1, \dots, a_{d-\delta+1}$ 
 are distinct  
elements  
of $\kk$, then 
\begin{equation}\label{28.9.2}\Gamma_{\phi_3}((\ell_1+a_i\ell_2)^{(d-\delta)}X\t \omega_U)=
\sum_{j=0}^{d-\delta}a_i^j\Gamma_{\phi_3}(\ell_1^{(d-\delta-j)}\ell_2^{(j)}X\t \omega_U).\end{equation}
The hypothesis guarantees that the left side of (\ref{28.9.2}) is zero. 
It follows that product of the row vector
{\small$$\bmatrix \Gamma_{\phi_3}(\ell_1^{(d-\delta-0)}\ell_2^{(0)}X\t \omega_U)&
\Gamma_{\phi_3}(\ell_1^{(d-\delta-1)}\ell_2^{(1)}X\t \omega_U)&
\dots &\Gamma_{\phi_3}(\ell_1^{(0)}\ell_2^{(d-\delta-0)}X\t \omega_U)\endbmatrix$$}
and the Vandermonde matrix
\begin{equation}\label{VT}\bmatrix 1&1&\dots&1\\a_1&a_2&\dots &a_{d-\delta+1}\\
\vdots&\vdots&&\vdots\\
a_1^{d-\delta}&a_2^{d-\delta}&\dots&a_{d-\delta+1}^{d-\delta}
\endbmatrix\end{equation} is zero.
The  Vandermonde matrix is invertible and $\Gamma_{\phi_3}(\ell_1^{(e_1)}\ell_2^{(e_2)}X\t \omega_U)=0$ for all $\ell_1$ and $\ell_2$ in $U$ and all non-negative integers $e_i$  with $e_1+e_2=d-\delta$.
This completes the proof of Claim~\ref{28.9.1}.

\ms
Now we prove Observation~\ref{28.9}. First take $X=1$. The hypothesis ensures that $\Gamma_{\phi_3}(\ell^{(d)}\t \omega_U)=0$ for all $\ell\in U$. Apply Claim~\ref{28.9.1}  
to conclude that
\begin{equation} \begin{array}{ll}\Gamma_{\phi_3}(\ell_1^{(e_1)}\ell_2^{(e_2)}\t \omega_U)=0\quad&\text{for all $\ell_1,\ell_2\in U$ and all nonnegative}\\&\text{integers $e_1$ and $e_2$ with $e_1+e_2=d$.}\end{array}\label{28.9.3}\end{equation}

Now take $X=\ell_3^{(e_3)}$ for some $\ell_3\in U$ and some integer $e_3$ with $0\le e_3\le d$. Apply Claim \ref{28.9.1}, together with (\ref{28.9.3}), to conclude that $\Gamma_{\phi_3}(\ell_1^{(e_1)}\ell_2^{(e_2)}\ell_3^{(e_3)}\t \omega_U)=0$, for all $\ell_1,\ell_2,\ell_3\in U$ and all nonnegative integers $e_1$, $e_2$, and $e_e$ with $e_1+e_2+e_3=d$.

One finishes the proof by iterating Claim~\ref{28.9.1}.
\end{proof}

\subsection{The connection between 
 $\Gamma_{\phi_3}$ and the weak Lefschetz property.}\label{connection}

\begin{lemma}
\label{28.10} Adopt Data~{\rm\ref{28.4}} and Definition~{\rm\ref{28.7}}. 
Assume $\ell\phi_3\neq 0$ for all nonzero $\ell$ in $U$. Then
$A=A_{\phi_3}$ has the weak Lefschetz property if and only if $\Gamma_{\phi_3}$ is not identically zero. 
\end{lemma}

\begin{remark}The hypothesis $\ell\phi_3\neq 0$ for all nonzero $\ell$ in $U$ is harmless. It is equivalent to asserting that the degree one component of the ideal $I$ is zero. Consequently, it is also equivalent to the hypothesis that the embedding dimension of $A$ is equal to the vector space dimension of $U$.
\end{remark}

\begin{proof} 
Recall from Observation~\ref{28.9} that 
$$\Gamma_{\phi_3}(\ell^{(d)}\t -)\text{ is zero for all $\ell\in U$}\iff \Gamma_{\phi_3} \text{ is
identically zero.}$$
Let $x_1,\dots,x_d$ be a basis for $U$,  $\ell$ be an element of $U$, and $\mu_\ell$ represent the homomorphism ``multiplication by $\ell$''. Observe that \begingroup\allowdisplaybreaks
\begin{align*} &\text{$\ell$ is a weak Lefschetz element in $A$}\\ 
{}\iff{}&\mu_\ell: A_1\to A_2 \text{ is injective}\\
{}\iff{}&\mu_\ell: P_1\to A_2 \text{ is injective,}&&\text{because $I_1=0$,}\\
{}\iff{}&\begin{cases}\ell(\sum a_ix_i)(\phi_3)=0,\text{ with $a_i$ in }\kk,\\
\text{only if all $a_i$ are zero,} \end{cases}\\
{}\iff{}&\begin{cases}
\ell x_1(\phi_3), \ell x_2(\phi_3),\dots,\ell x_d(\phi_3) \text{ are linearly}\\\text{independent in $U^*$,}\end{cases}\\
{}\iff{}&\Gamma_{\phi_3}(\ell^{(d)}\t x_1\w x_2\w\dots \w x_d)\neq 0.\\
\end{align*}\endgroup
\vskip-35pt
\end{proof}

Theorem~\ref{main}, which is the main result of this paper, is an immediate consequence of Lemma~\ref{ML}, which is an immediate consequence of Lemma~\ref{28.10}.

\begin{lemma}
\label{ML} Adopt Data~{\rm\ref{28.4}} with $d=4$. Assume
\begin{enumerate}
[\rm(a)]
\item\label{ML.a} either the characteristic of $\kk$ is different than two; or else, the characteristic of $\kk$ is equal to two, but there does not exist a basis $x^*$, $y^*$, $z^*$, $w^*$ for $U^*$ with $\phi_3=x^{*(3)}+y^*z^*w^*$, and \item\label{ML.b} $\ell\phi_3\neq 0$ for all nonzero $\ell\in U$.\end{enumerate}
Then $\Gamma_{\phi_3}$ is not identically zero.\end{lemma}  

The proof of Lemma~\ref{ML} involves multiple cases and comprises the majority of this paper. The official proof is given in (\ref{proof of ML}).

\section{Put the Macaulay inverse system into a convenient form.}\label{convenient}

Ultimately, we prove a statement about an element $\phi_3$ of $D_3U^*$, where $U$ is a four-dimensional vector space. Our proof depends on the form of $\phi_3$. There are four main cases. Two of the cases involve $\phi_3$ as described in (\ref{61.1.a}) of Lemma~\ref{61.1}. (These two cases are distinguished by the rank $r$ of the homomorphism $p_{\phi_{2,0}}$; see Lemma~\ref{55.Claim 3}.) These two cases are treated in Propositions~\ref{58.2} and \ref{5.2}. The other two cases involve $\phi_3$ as described in (\ref{61.1.b}) of Lemma~\ref{61.1}. These two cases are separated in \ref{Two Cases} and are treated in Propositions~\ref{6.1} and \ref{NYW}.

In most characteristics, all $\phi_3$ can be put in the form of (\ref{61.1.a}) of Lemma~\ref{61.1}; however, form (\ref{61.1.b}) of Lemma~\ref{61.1} is required in characteristic three.

\begin{lemma}
\label{61.1}Let $\kk$ be a field, $U$ be a $d$-dimensional vector space 
over $\kk$, and $\phi_3$ be a nonzero element of $D_3U^*$. Then there exists a unit $\a$ of $\kk$ and a basis $x_1^*,\dots,x_d^*$ for $U^*$ 
such that
\begin{enumerate}[\rm(a)]
\item\label{61.1.a}
$\a \phi_3=x_1^{*(3)}+x_1^*\phi_{2,0}+\phi_{3,0}$, or 
\item\label{61.1.b} $\a \phi_3=x_1^{*(2)}x_2^*+x_1^*\phi_{2,0}+\phi_{3,0}$,  
or
\item\label{61.1.c}$\kk$ has characteristic two and  
$$\phi_3=\sum_{1\le i<j<k\le d}\a_{i,j,k}x_i^*x_j^*x_k^*$$ for some $\a_{i,j,k}$ in $\kk$,
\end{enumerate} 
where $\phi_{i,0}$ is an element of $$D_i(\bigoplus_{2\le j\le d}\kk x_j^*).$$
\end{lemma}

\begin{Remark}
The case (\ref{61.1.c}) is not very interesting when $d=4$; see (\ref{proof of ML}).
\end{Remark}

\begin{proof}
 Begin with an arbitrary basis 
$y_1^*,\dots,y_d^*$ for $U^*$.

\begin{claim-no-advance}
\label{5.Claim 1}  
After a  change of basis, 
\begingroup\allowdisplaybreaks\begin{align}\label{5.form}\a\phi_3={}&x_1^{*(3)}+x_1^{*(2)}\phi_{1,0}+x_1^{*}\phi_{2,0}+\phi_{3,0}, \text{ with $\a\neq 0$, or}\\
\label{5.form2}\phi_3={}&\phantom{x_1^{*(3)}+{}}x_1^{*(2)}\phi_{1,0}+x_1^{*}\phi_{2,0}+\phi_{3,0},\text{ with $\phi_{1,0}\neq 0$, or}\\
\label{5.form3}\phi_3={}&\sum_{1\le i<j<k\le d}\a_{i,j,k}x_i^*x_j^*x_k^*,
\end{align}\endgroup 
with $\a,\a_{i,j,k}$ in $\kk$ and $\phi_{i,0}\in D_i(\bigoplus_{2\le j\le d}\kk x_j^*)$.\end{claim-no-advance}

\ms\noindent{\bf Proof of Claim~\ref{5.Claim 1}.} 
Write $\phi_3$ as 
$$\phi_3=\sum\limits_{e_1+\dots+e_d=3}\a_{e_1,\dots,e_d}y_1^{*(e_1)}
y_2^{*(e_2)}
\cdots
y_{d-1}^{*(e_{d-1})}y_d^{*(e_d)},$$ with $\a_{e_1\dots,e_d}\in \kk$.
If any of the parameters \begin{equation}\label{5.10.2}
\a_{0,\dots,0,3,0,\dots,0} 
\end{equation}
 is nonzero, then $\phi_3$ has the form of (\ref{5.form}). 

If the parameters of (\ref{5.10.2}) are zero; but any of the parameters
\begin{equation}\label{5.10.3}
\a_{0,\dots,0,2,0,\dots,0,1,0\dots,0}\end{equation}
are nonzero, then $\phi_3$ has the form of (\ref{5.form2}). 
If all of the parameters in (\ref{5.10.2}) and (\ref{5.10.3}) are zero, then $\phi_3$ has the form of (\ref{5.form3}). This completes the proof of Claim~\ref{5.Claim 1}.

\ms

Three observations are needed to complete the proof of Lemma~\ref{61.1}. First, if $\phi_3$ has the form of (\ref{5.form}), then the change of basis $x_1^*=X^*-\phi_{1,0}$ puts $\phi_3$ into the form of (\ref{61.1.a}) because 
$$(X^*-\phi_{1,0})^{(3)}=X^{*(3)}-X^{*(2)}\phi_{1,0}+X^*\phi_{1,0}^{(2)}-\phi_{1,0}^{(3)}.$$
Second, if $\phi_3$ has the form of
(\ref{5.form2}), then one may change the basis of $U^*$ again and choose the new ``$x_2^*$'' to equal the old $\phi_{1,0}$.
Third, 
if the characteristic of $\kk$ is not two, then any nonzero element of $D_3U^*$ of form (\ref{5.form3}) can be transformed into an element of $D_3U^*$ of form (\ref{5.form2}). In particular,  an element of form (\ref{5.form3}) in which
$x_1^*x_2^*x_3^*$ actually appears 
 becomes an element of form (\ref{5.form2})
if one uses the basis $x_1^*, y_2^*, x_3^*,\dots, x_d^*$  
for $U^*$ with $x_2^*=y_2^*+x_1^*$ because
$$x_1^*x_1^*=2x_1^{*(2)}$$ and this is a unit times $x_1^{*(2)}$ when the characteristic of $\kk$ is 
 not two. \end{proof}

Lemma~\ref{55.Claim 3} ensures that we can record ``$\phi_{2,0}$'' from Lemma~\ref{61.1}.(\ref{61.1.a}) in an efficient manner.
\begin{lemma}
\label{55.Claim 3}
Let $U_0$ be a $d_0$-dimensional vector space over the field $\kk$ and $\phi_{2,0}$ be an element of $D_2U_0^*$. Let $p_{\phi_{2,0}}:U_0\to U_0^*$ be the homomorphism defined by $p_{\phi_{2,0}}(\ell)=\ell\phi_{2,0}$ and let $r$  be the rank of $p_{\phi_{2,0}}$. Then there is a basis $\ell_1, \dots,\ell_{d_0}$ for $U_0$ and corresponding dual basis $\ell_1^*, \dots,\ell_{d_0}^*$ for $U_0^*$ such that
$\phi_{2,0}\in D_2(\kk\ell_1^*\p\dots\p\kk\ell_r^*)$. In particular, if $r=1$, then $\phi_{2,0}=a\ell_1^{*(2)}$ for some nonzero element $a$ in $\kk$; and if $r=2$, then \begin{equation}\label{3.2.1}\phi_{2,0}=a\ell_1^{*(2)}+b\ell_1^*\ell_2^*+c\ell_2^{*(2)},\end{equation}
for some elements $a,b,c$ of $\kk$ with $ac-b^2$ not equal to zero.\end{lemma}

\begin{proof}
 Let $\ell_1,\dots,\ell_{d_0}$ be a basis for  $U_0$ such that $\ell_1\phi_{2,0},\dots,\ell_r\phi_{2,0}$ is a basis for the image of $p_{\phi_{2,0}}$ and $\ell_{r+1},\dots,\ell_{d_0}$ are in $\ker p_{\phi_{2,0}}$. Let $\ell_1^*, \dots,\ell_{d_0}^*$ be the corresponding dual basis for $U_0^*$. Write $\phi_{2,0}$ in terms of the basis $$\{\ell_1^{*(i_1)}\dots\ell_{d_0}^{*(i_{d_0})}\mid \sum_j i_j=2\}$$ for $D_2U_0^*$. The hypothesis that $\ell_h\phi_{2,0}=0$ for $r+1\le h\le 3$ ensures that  $\phi_{2,0}$ is in $D_2(\kk\ell_1^*\p\dots\p\kk\ell_r^*)$.  Furthermore, if $r=1$, then the coefficient of $\ell_1^{*(2)}$ can not be zero; and if $r=2$, then $\ell_1\phi_{2,0}$ and $\ell_2\phi_{2,0}$ must be linearly independent (hence, $\phi_{2,0}$ must have the form of (\ref{3.2.1}) with $ac-b^2\neq 0$.) \end{proof}

Lemma~\ref{3.3} is redundant in the sense that the assertion is a consequence of Lemma~\ref{55.Claim 3} when $r=1$. On the other hand, the constructive nature of the argument makes this statement a valuable  addition to one's tool bag.

\begin{lemma}
\label{3.3} 
Let $\kk$ be a field, $U_0^*$ be a two-dimensional vector space over $\kk$, and $\phi_{2,0}\in D_2U_0^*$. If $\dim \{\ell\phi_{2,0}\mid \ell\in U_0\}$ is at most one, then there is a basis $z^*,w^*$ for $U_0^*$ such that $\phi_{2,0}=az^{*(2)}$ for some $a\in \kk$.\end{lemma}
\begin{proof} 
Let $Z^*,W^*$ be a basis for $U_0^*$ and $Z,W$ be the corresponding basis for $U_0$. Write $\phi_{2,0}=aZ^{*(2)}+bZ^*W^*+cW^{*(2)}$ for some $a,b,c\in \kk$.
The hypothesis ensures that $$Z\phi_{2,0}=aZ^*+bW^*\quad\text{and}\quad W\phi_{2,0}=bZ^*+cW^*$$ are linearly dependent. It follows that $ac=b^2$.

If $a$, $b$, and $c$ are all zero, then the conclusion  holds automatically. Henceforth, we assume that at least one of the parameters $a$, $b$, or $c$ is nonzero; so, in particular, $a$ or $c$ is nonzero.  Without loss of generality, we assume $a\neq 0$. In this case, 
$$\ts aZ^{*(2)}+bZ^*W^*+cW^{*(2)}=a(Z^{*(2)}+\frac baZ^*W^*+\frac {b^2}{a^2}W^{*(2)})=a(Z^*+\frac ba W^*)^{(2)}$$because $\frac{b^2}{a^2}=\frac ca$. At this point, we rename $z^*= Z^*+\frac ba W^*$ and $w^*=W^*$. The claim is  established.
\end{proof}

\begin{chunk}
{\bf The proof of Lemma~\ref{ML}.}
\label{proof of ML}There are many cases depending upon the form of $\phi_3$ in the sense of Lemma~\ref{61.1}. 

If $\phi_3$ has the form of Lemma~\ref{61.1}.(\ref{61.1.a}), then let $r$ be the rank of $p_{\phi_{2,0}}$ as described in Lemma~\ref{55.Claim 3}. Of course, $0\le r\le 3$. If $r=2$, then Lemma~\ref{ML} is established in Proposition~\ref{58.2}; if $r=1$, then Lemma~\ref{ML} is established in Proposition~\ref{5.2}; and if $r$ is either $0$ or $3$, then Lemma~\ref{ML} is established in Lemma~\ref{case3a}.

If $\phi_3$ has the form of Lemma~\ref{61.1}.(\ref{61.1.b}), then there are two cases as described in \ref{Two Cases}. Case \ref{Case1} of \ref{Two Cases} 
is established in Proposition~\ref{NYW} and Case \ref{Case2}  is established in Proposition~\ref{6.1}.

If $\phi_3$ has form of Lemma~\ref{61.1}.(\ref{61.1.c}), then $\kk$ has characteristic two and $$\phi_3=ay^*z^*w^*+bx^*z^*w^*+cx^*y^*w^*+dx^*y^*z^*$$ for some parameters $a,b,c,d$ from $\kk$. If $\ell=ax^*+by^*+cz^*+dw^*$, then $$\ell\phi_3=2(abz^*w^*+acy^*w^*+ady^*z^*+bcx^*w^*+bdx^*z^*+cdx^*y^*)=0.$$ Thus, $\ell\phi_3=0$ for
some nonzero $\ell\in U$. In this case, Lemma~\ref{ML} makes no assertion about $\Gamma_{\phi_3}$. 
\qed
\end{chunk}

\section{The Macaulay inverse system has a cubic term and $r=2$.}\label{cubic,r=2}

We prove Lemma~\ref{ML} when $\phi_3$ has the form of Lemma~\ref{61.1}.(\ref{61.1.a}), with $r=2$, in the sense of Lemma~\ref{55.Claim 3}. This means that, in the language of Data~\ref{28.4},  there is a basis $x^*$, $y^*$, $z^*$, $w^*$ for $U^*$ so that $\phi_3=x^{*(3)}+\phi_{2,0}x^*+\phi_{3,0}$ with $\phi_{2,0}\in D_2 \kk(z^*,w^*)$, $\phi_{3,0}\in D_3 \kk(y^*,z^*,w^*)$, and $z\phi_{2,0}$ and $w\phi_{2,0}$ linearly independent. In particular, the basis elements $x^{*(2)}y^*$,  $x^{*(2)}z^*$, $x^{*(2)}w^*$, $x^*y^{*(2)}$,  $x^*y^*z^*$, and $x^*y^*w^*$ of $D_3U^*$ all appear in $\phi_3$ with coefficient zero.\begin{proposition}
\label{58.2} 
Let $P_0$ be a domain, $U$ be a free $P_0$-module of rank four with dual module $U^*=\Hom_{P_0}(U,P_0)$, $x,y,z,w$ be a basis for $U$ with dual basis $x^*,y^*,z^*,w^*$ for $U^*$. Let $\phi_3\in D_3U^*$ have the form
\begin{equation}\label{58.2.1}\phi_3=\begin{cases}\phantom{+}x^{*(3)}+ax^*z^{*(2)}+bx^*z^*w^*+cx^*w^{*(2)}+dy^*z^{*(2)}\\+ey^*z^*w^*+
fy^*w^{*(2)}+gy^{*(2)}z^*+hy^{*(2)}w^*+iy^{*(3)}\\+jz^{*(3)}+kz^{*(2)}w^*+lz^*w^{*(2)}+mw^{*(3)},\end{cases}\end{equation} for elements $a,\dots,m$ of $P_0$. Assume that 
\begin{enumerate}[\rm(a)]\item$ac-b^2\neq 0$ in $P_0$,
\item$\ell\phi_3\neq 0$ for all nonzero $\ell$ in $U$, and \item $\Gamma_{\phi_3}=0$. \end{enumerate}
 Then $2$ is equal to $0$ in $P_0$, and, after a change of basis, $\phi_3=\a X^{*(3)}+Y^*Z^*W^*$ for some basis $X^*$, $Y^*$, $Z^*$ of $U^*$ and some nonzero $\a\in P_0$.\end{proposition}

\begin{proof} 
Let $\omega_U$ and $\omega_{U^*}$ represent the bases
$x\w y\w z\w w$ and  $x^*\w y^*\w z^*\w w^*$ of $\bw^4U$ and $\bw^4U^*$, respectively.

\begin{claim-no-advance}
\label{58.2.2} The parameters $g$, $h$, and $i$ are zero.\end{claim-no-advance} 

\ms \noindent{\bf Proof of Claim \ref{58.2.2}.} We first compute 
\begin{align}\Gamma_{\phi_3}(x^{(3)}y\t \omega_U)={}&i(ac-b^2)\omega_{U^*},\notag\\
\Gamma_{\phi_3}(x^{(3)}z\t \omega_U)={}&g(ac-b^2)\omega_{U^*},\text{ and}\label{58.2.3}\\
\Gamma_{\phi_3}(x^{(3)}w\t \omega_U)={}&h(ac-b^2)\omega_{U^*}.\notag\end{align}
Let $\ell$ be an arbitrary element of $U$. The expansion of $\Gamma_{\phi_3}(x^{(3)}\ell\t \omega_U)$ has four summands; but three of the summands involve a factor of $xy\phi_3=0$; consequently,
\begingroup\allowdisplaybreaks\begin{align*}&\Gamma_{\phi_3}(x^{(3)}\ell\t \omega_U)\\{}={}&x^2\phi_3\w y\ell\phi_3\w xz\phi_3\w xw\phi_3
=x^*\w y\ell\phi_3\w (az^*+bw^*)\w (bz^*+cw^*)\\
{}={}&(ac-b^2) x^*\w y\ell\phi_3\w z^*\w w^*.\end{align*}\endgroup
Insert 
\begingroup\allowdisplaybreaks\begin{align*}
y^2\phi_3={}&(gz^*+hw^*+iy^*)\\
yz\phi_3={}&(dz^*+ew^*+gy^*)\\
yw\phi_3={}&(ez^*+fw^*+hy^*)\end{align*}\endgroup
into the calculation of $\Gamma_{\phi_3}(x^{(3)}\ell\t \omega_U)$ in order to obtain
(\ref{58.2.3}). Combine the hypotheses that $\Gamma_{\phi_3}$ is identically zero, but $ac-b^2$ is nonzero in the domain $P_0$ to conclude that $g=h=i=0$.
This completes the proof of Claim~\ref{58.2.2}.

\ms

\begin{claim-no-advance}
\label{58.2.4} The parameter $d$ is zero.\end{claim-no-advance} 

\ms \noindent{\bf Proof of Claim \ref{58.2.4}.} We first compute that if
$g=h=i=0$, then 
\begin{align}
\Gamma_{\phi_3}(x^{(2)}z^{(2)}\t \omega_U)={}&(-cd^2+2bde-ae^2)\omega_{U^*}\label{58.2.5}\\
\Gamma_{\phi_3}(z^{(4)}\t \omega_U)={}&(ae-bd)^2 \omega_{U^*}\notag\end{align}
The expansion of $\Gamma_{\phi_3}(x^{(2)}z^{(2)}\t \omega_U)$ consists of six summands; four of the summands have a factor of $xy\phi_3=0$ or $xz\phi_3\w xz\phi_3=0$. It follows that 
\begingroup\allowdisplaybreaks\begin{align*}
&\Gamma_{\phi_3}(x^{(2)}z^{(2)}\t \omega_U)\\{}={}&
x^2\phi_3\w yz\phi_3\w xz\phi_3\w zw\phi_3+
x^2\phi_3\w yz\phi_3\w z^2\phi_3\w xw\phi_3\\
{}={}&\begin{cases}\phantom{+}
 x^*\w(dz^*+ew^*)\w(az^*+bw^*)\w(bx^*+ey^*+kz^*+lw^*)\\
+x^*\w(dz^*+ew^*)\w(ax^*+dy^*+jz^*+kw^*)\w (bz^*+cw^*)\end{cases}\\
{}={}&\Big(e(bd-ae)-d(cd-eb)\Big)\omega_{U^*}\\
{}={}&(-cd^2+2bde-ae^2)\omega_{U^*}.\end{align*}\endgroup
Also, one computes  
\begingroup\allowdisplaybreaks\begin{align*}
&\Gamma_{\phi_3}(z^{(4)}\t \omega_U)\\
{}={}&
xz\phi_3\w yz\phi_3\w z^2\phi_3\w zw\phi_3\\
{}={}&(az^*+bw^*)\w (dz^*+ew^*)\w (ax^*+dy^*+jz^*+kw^*)\w(bx^*+ey^*+kz^*+lw^*)\\
{}={}&(ae-bd)(z^*\w w^*)\w (ae-bd)(x^*\w y^*)\\
{}={}&(ae-bd)^2\omega_{U^*}.
\end{align*}\endgroup
Both formulas of (\ref{58.2.5}) have been established. The hypothesis that $\Gamma_{\phi_3}=0$ ensures that
$$0=-cd^2+2bde-ae^2\quad \text{and}\quad  0=(ae-bd)^2.$$
It follows that
$$0=a(-cd^2+2bde-ae^2)+(ae-bd)^2=d^2(b^2-ac).$$
The ring $P_0$ is a domain and $b^2-ac$ is not zero. We conclude that $d=0$. This completes the proof of Claim~\ref{58.2.4}.

\ms
In Lemma~\ref{calculation} we prove that if $d=g=h=i=0$, then
\begingroup\allowdisplaybreaks\begin{align}
\Gamma_{\phi_3}(xz^{(3)}\t \omega_U)={}&-e^2j \omega_{U^*},\label{first}\\
\Gamma_{\phi_3}(xyzw\t \omega_U)={}&2e^3 \omega_{U^*},\label{second}\\
\Gamma_{\phi_3}(xz^{(2)}w\t \omega_U)={}&(e^2k-2efj) \omega_{U^*},\label{third}\\
\Gamma_{\phi_3}(x^{(2)}w^{(2)}\t \omega_U)={}&(-ce^2+2bef-af^2) \omega_{U^*},\label{fourth}\\
\Gamma_{\phi_3}(xyw^{(2)}\t \omega_U)={}&e^2f \omega_{U^*},\label{fifth}\\
\Gamma_{\phi_3}(xzw^{(2)}\t \omega_U)={}&(e^2l-jf^2) \omega_{U^*},\label{sixth}\\
\Gamma_{\phi_3}(z^{(2)}w^{(2)}\t \omega_U)={}&(-2ace^2+2abef+a^2f^2) \omega_{U^*},\label{seventh}\\
\Gamma_{\phi_3}(xw^{(3)}\t \omega_U)={}&(-f^2k+2efl-e^2m) \omega_{U^*},\text{ and}\label{eighth}\\
\Gamma_{\phi_3}(w^{(4)}\t \omega_U)={}&(bf-ce)^2 \omega_{U^*}.\label{last}
\end{align}\endgroup

\begin{claim-no-advance}
\label{58.2.15}
The following assertions hold{\rm:}
\begin{enumerate}[\rm(a)]
\item\label{58.a}$ae=0$, \item\label{58.b}$ej=0$,\item\label{58.c} $2e=0$,\item\label{58.d} 
$ek=0$,\item\label{e} $ce^2+af^2=0$, \  \item\label{f}$ef=0$,\ \item\label{g}$-f^2j+e^2\ell=0$,\ \item\label{h}$af=0$,\ \item\label{i}$-f^2k-e^2m=0$, \ \text{and}\ \item\label{j}$ce-bf=0.$
\end{enumerate}\end{claim-no-advance}

\ms\noindent{\bf Proof of Claim~\ref{58.2.15}.} In the proof of Claim~\ref{58.2.4}, we saw that $(ae-bd)^2=0$ and $d=0$. The ring $P_0$ is a domain. It follows that $ae=0$.
Assertion~(\ref{58.b}) is a consequence of (\ref{first}) and the hypothesis that $P_0$ is a domain; (\ref{58.c}) follows from (\ref{second}); (\ref{58.d}) from (\ref{third}) and (\ref{58.c}); (\ref{e}) from (\ref{fourth}) and (\ref{58.c});  (\ref{f}) from (\ref{fifth}); (\ref{g}) from (\ref{sixth}); (\ref{h}) from (\ref{seventh}) and (\ref{58.c}); (\ref{i}) from (\ref{eighth}) and (\ref{58.c}); and (\ref{j}) from (\ref{last}). This completes the proof of Claim~\ref{58.2.15}.

\begin{claim-no-advance}
\label{58.2.16} The parameter $e$ is not equal to zero.\end{claim-no-advance} 

\noindent{\bf  Proof of Claim~\ref{58.2.16}.} This proof is by contradiction: suppose $e=0$. There are two cases: either $f=0$ or $f\neq 0$.

If $f=0$, then, 
$d=e=f=g=h=i=0$, $\phi_3$ is equal to
$$x^{*(3)}+ax^*z^{*(2)}+bx^*z^*w^*+cx^*w^{*(2)}
+jz^{*(3)}+kz^{*(2)}w^*+lz^*w^{*(2)}+mw^{*(3)},$$
and $y\phi_3=0$, which violates the hypothesis that $\ell\phi_3$ is nonzero for all nonzero $\ell$ in $U$. 

If $e=0$ and $f\neq 0$, then, according to Claim~\ref{58.2.15}, $af=bf=0$.
Thus, $$a=b=0;$$ however, the ambient hypothesis  guarantees that $b^2-ac\neq 0$.  
This completes the proof of Claim~\ref{58.2.16}.

\ms 
Use the fact that $e\neq0$, together with  
Claims~\ref{58.2.15}, \ref{58.2.4}, and \ref{58.2.2} to see that $$2=a=c=d=f=g=h=i=j=k=l=m=0.$$
In this case two is equal to zero in $P_0$ and
$$\phi_3=
x^{*(3)}
+bx^*z^*w^*+ey^*z^*w^*=x^{*(3)}+(bx^*+ey^*)z^*w^*,$$ 
which has the form $\phi_3=X^{*(3)}+Y^*Z^*W^*$ for some basis $X^*$, $Y^*$, $Z^*$, $W^*$ of $U^*$. \end{proof}

\begin{lemma}
\label{calculation}
If $\phi_3$ is given in {\rm(\ref{58.2.1})} with $d = g = h = i = 0$, then
 the assertions of {\rm (\ref{first})} to {\rm(\ref{last})} all hold. 
\end{lemma}

\begin{proof} 
We prove (\ref{first}). The expansion of 
$\Gamma_{\phi_3}(xz^{(3)}\t \omega_U)$ consists of four summands. Three of the summands  contain a factor of $xy\phi_3=0$, or $xz\phi_3\w xz\phi_3=0$, or
\begin{equation}\label{rank2}\ts\bw^3 \kk(xw\phi_3,\ yz\phi_3,\ xz\phi_3,\ yw\phi_3)
\subseteq \bw^3(P_0 z^*\p P_0 w^*)=0.\end{equation}
It follows that 
\begingroup\allowdisplaybreaks\begin{align*}
\Gamma_{\phi_3}(xz^{(3)}\t \omega_U)={}& x^2\phi_3\w yz\phi_3\w z^2\phi_3\w zw\phi_3\\
{}={}&x^*\w ew^*\w(ax^*+jz^*+kw^*)\w (bx^*+ey^*+kz^*+lw^*)\\
{}={}&x^*\w ew^*\w jz^*\w ey^*=-e^2j\omega_{U^*}.\end{align*}\endgroup

\ms We prove  (\ref{second}). The expansion of 
$\Gamma_{\phi_3}(xyzw\t \omega_U)$ consists of twenty-four summands. Twenty-two of the summands have a factor of $$\begin{gathered}y^2\phi_3=0,\ \   xy\phi_3=0,\ \    yz\phi_3\w yz\phi_3=0,\ \    xz\phi_3\w xz\phi_3=0,\ \  xw\phi_3\w xw\phi_3=0,\  \ \text{or}\\yw\phi_3\w yw\phi_3,\ \ \text{or}\ \  
(\ref{rank2}).\end{gathered}$$
The other two summands are equal. Thus, 
\begingroup\allowdisplaybreaks\begin{align*}\Gamma_{\phi_3}(xyzw\t \omega_U)={}&2x^2\phi_3\w yw\phi_3\w yz\phi_3\w zw\phi_3 \\ 
{}={}&2x^*\w (ez^*+ fw^*)\w ew^*\w(bx^*+ey^*+kz^*+lw^*)\\
{}={}&2x^*\w ez^*\w ew^*\w ey^*\\
{}={}&2e^3\omega_{U^*}.\end{align*}\endgroup

\ms We prove  (\ref{third}).
The expansion of $\Gamma_{\phi_3}(xz^{(2)}w\t \omega_U)$ consists of twelve summands. Ten of the summands contain a factor of $$xy\phi_3=0,\quad zw\phi_3\w zw\phi_3=0,\quad xz\phi_3\w xz\phi_3=0,\quad  
xw\phi_3\w xw\phi_3=0, \quad\text{(\ref{rank2})},\quad\text{or}$$ 
\begin{equation}\label{wedge4rank3}\ts
\bw^4\kk(xz\phi_3,\ yw\phi_3,\  z^2\phi_3,\  xw\phi_3)\subseteq
\bw^4(P_0x^*\p P_0 z^*\p P_0w^*)=0.\end{equation} 
It follows that
\begingroup\allowdisplaybreaks\begin{align*}&\Gamma_{\phi_3}(xz^{(2)}w\t \omega_U)\\{}={}&x^2\phi_3\w yw\phi_3\w z^2\phi_3\w zw\phi_3+x^2\phi_3\w yz\phi_3\w z^2\phi_3\w w^2\phi_3\\
{}={}&\begin{cases}\phantom{+} x^*\w (ez^*+fw^*)\w(ax^*+jz^*+kw^*)\w (bx^*+ey^*+kz^*+lw^*)\\
+x^*\w ew^*\w (ax^*+jz^*+kw^*)\w (cx^*+fy^*+lz^*+mw^*)\end{cases}
\\{}={}&(e^2k-2efj)\omega_{U^*}.\end{align*}\endgroup

\ms We prove  (\ref{fourth}).
The expansion of $\Gamma_{\phi_3}(x^{(2)}w^{(2)}\t \omega_U)$ consists of six summands. Four of the summands have a factor of $xy\phi_3=0$ or $xw\phi_3\w xw\phi_3=0$. Thus,
 \begingroup\allowdisplaybreaks\begin{align*}&\Gamma_{\phi_3}(x^{(2)}w^{(2)}\t \omega_U)\\{}={}&
x^2\phi_3\w yw\phi_3\w xz\phi_3\w w^2\phi_3+
x^2\phi_3\w yw\phi_3\w zw\phi_3\w xw\phi_3\\
{}={}&\begin{cases}\phantom{+}
 x^*\w (ez^*+fw^*)\w(az^*+bw^*)\w (cx^*+fy^*+lz^*+mw^*)\\
+x^*\w (ez^*+fw^*)\w(bx^*+ey^*+kz^*+lw^*)\w(bz^*+cw^*)\end{cases}\\
{}={}&\Big(f(eb-fa)-e(ec-bf)\Big)\omega_{U^*}=(-ce^2+2bef-af^2) \omega_{U^*}.\end{align*}\endgroup

\ms We prove  (\ref{fifth}).
The expansion of $\Gamma_{\phi_3}(xyw^{(2)}\t \omega_U)$ consists of twelve summands. Eleven  of the summands have a factor of
$$y^2\phi_3=0,\quad yw\phi_3\w yw\phi_3=0, \quad xy\phi_3=0,\quad\text{or}\quad xw\phi_3\w xw\phi_3=0.$$
It follows that
\begingroup\allowdisplaybreaks\begin{align*}\Gamma_{\phi_3}(xyw^{(2)}\t \omega_U)={}&
x^2\phi_3\w yw\phi_3\w yz\phi_3\w w^2\phi_3\\
{}={}&x^*\w (ez^*+fw^*)\w ew^*\w(cx^*+fy^*+lz^*+mw^*)\\
{}={}&e^2f\omega_{U^*}\end{align*}\endgroup

\ms We prove  (\ref{sixth}).
The expansion of $\Gamma_{\phi_3}(xzw^{(2)}\t \omega_U)$ consists of twelve summands. Ten   of the summands have a factor of
$$xy\phi_3=0, \quad zw\phi_3\w zw\phi_3=0,\quad xz\phi_3\w xz\phi_3=0, \quad xw\phi_3\w xw\phi_3=0\quad \text{or}\quad(\ref{rank2}).$$
It follows that 
\begingroup\allowdisplaybreaks\begin{align*}&\Gamma_{\phi_3}(xzw^{(2)}\t \omega_U)\\
{}={}&x^2\phi_3\w yz\phi_3\w zw\phi_3\w w^2\phi_3+x^2\phi_3\w yw\phi_3\w z^2\phi_3\w w^2\phi_3\\
{}={}&\begin{cases}\phantom{+}
 x^*\w ew^*\w (bx^*+ey^*+kz^*+lw^*)\w (cx^*+fy^*+lz^*+mw^*)\\
+x^*\w(ez^*+fw^*)\w(ax^*+jz^*+kw^*)\w(cx^*+fy^*+lz^*+mw^*)\end{cases}\\
{}={}&\Big(e(el-fk)+f(ek-jf) \Big) \omega_{U^*}=(e^2\ell-jf^2)\omega){U^*}.
\end{align*}\endgroup

\ms We prove  (\ref{seventh}).
The expansion of $\Gamma_{\phi_3}(z^{(2)}w^{(2)}\t \omega_U)$ consists of six summands. Two   of the summands have a factor of
$zw\phi_3\w zw\phi_3=0$.
It follows that 
\begingroup\allowdisplaybreaks\begin{align*}&\Gamma_{\phi_3}(z^{(2)}w^{(2)}\t \omega_U)\\
{}={}&
\begin{cases}\phantom{+}
 xz\phi_3\w yz\phi_3\w zw\phi_3\w w^2\phi_3\\
+xz\phi_3\w yw\phi_3\w z^2\phi_3\w w^2\phi_3\\
+xw\phi_3\w yz\phi_3\w z^2\phi_3\w w^2\phi_3\\
+xw\phi_3\w yw\phi_3\w z^2\phi_3\w zw\phi_3\\
\end{cases}\\
{}={}&\begin{cases}\phantom{+}
 (az^*+bw^*)\w ew^*\w(bx^*+ey^*+kz^*+lw^*)\w(cx^*+fy^*+lz^*+mw^*)\\
+(az^*+bw^*)\w (ez^*+fw^*)\w(ax^*+jz^*+kw^*)\w(cx^*+fy^*+lz^*+mw^*)\\
+(bz^*+cw^*)\w ew^*\w(ax^*+jz^*+kw^*)\w(cx^*+fy^*+lz^*+mw^*)\\
+(bz^*+cw^*)\w(ez^*+fw^*)\w(ax^*+jz^*+kw^*)\w(bx^*+ey^*+kz^*+lw^*)
\end{cases}\\
{}={}&\Big(ae(bf-ec)+af(af-be)+afbe+ae(bf-ce) \Big) \omega_{U^*}\\
{}={}&(2abef-2ace^2+a^2f^2)\omega){U^*}.
\end{align*}\endgroup

\ms We prove  (\ref{eighth}).
The expansion of $\Gamma_{\phi_3}(xw^{(3)}\t \omega_U)$ consists of four summands. Three   of the summands have a factor of
$xw\phi_3\w xw\phi_3=0$, $xy\phi_3=0$, or (\ref{rank2}).
It follows that 
\begingroup\allowdisplaybreaks\begin{align*}&\Gamma_{\phi_3}(xw^{(3)}\t \omega_U)\\
{}={}& x^2\phi_3\w yw\phi_3\w zw\phi_3\w w^2\phi_3\\
{}={}&x^*\w (ez^*+fw^*)\w (bx^*+ey^*+kz^*+lw^*)\w (cx^*+fy^*+lz^*+mw^*)\\
{}={}&\left|\begin{matrix}0&e&f\\e&k&l\\f&l&m \end{matrix}\right| \omega_{U^*}\\
{}={}&(-f^2k+2efl-e^2m) \omega_{U^*}.
\end{align*}\endgroup

\ms We prove  (\ref{last}).
Observe  that 
\begingroup\allowdisplaybreaks\begin{align*}\Gamma_{\phi_3}(w^{(4)}\t \omega_U)
={}& xw\phi_3\w yw\phi_3\w zw\phi_3\w w^2\phi_3\\
{}={}&\begin{cases}(bz^*+cw^*)\w (ez^*+fw^*)\w (bx^*+ey^*+kz^*+lw^*)\\
\w (cx^*+fy^*+lz^*+mw^*)\end{cases}\\
{}={}&(bf-ce)z^*\w w^* \w(bf-ce)x^*\w y^*\\
{}={}&(bf-ce)^2\omega_{U^*}.
\end{align*}\endgroup
\vskip-24pt\end{proof}

\section{The Macaulay inverse system has a cubic term and $r=1$.}\label{cubic,r=1}

In this section 
$\phi_3=x^{*(3)}+\phi_{2,0}x^*+\phi_{3,0}$, with $\phi_{i,0}\in D_iU_0^*$, and the rank of $$p_{\phi_{2,0}}:U_0\to U_0^*$$
is equal to one, where  $U^*=\kk x^*\p U_0^*$ and $p_{\phi_{2,0}}(\ell_0)=\ell_0(\phi_{2,0})\in U_0^*$, for $\ell_0\in U_0$. We prove that if $\ell\phi_3$ is nonzero for all nonzero $\ell$ in $U$, then $\Gamma_{\phi_3}$ is not identically zero (and, therefore, $A_{\phi_3}$ has the weak Lefschetz property by 
Lemma~\ref{28.10}.) 

According to Lemma~\ref{55.Claim 3}, the hypothesis about the rank of $p_{\phi_{2,0}}$ ensures  that there is a basis $y^*,z^*,w^*$ for $U^*$ such that $\phi_{2,0}=az^{*(2)}$ for some unit $a$ in $\kk$. In particular, the basis elements $x^{*(2)}y^*$, $x^{*(2)}z^*$, $x^{*(2)}w^*$, $x^*y^{*(2)}$, $x^*y^*z^*$, $x^*y^*w^*$, $x^*z^{*(2)}$, $x^*z^*w^*$, and $x^*w^{*(2)}$
of $D_3U^*$ appear in $\phi_3$ with coefficient zero. Thus, $\phi_3$ has the form of (\ref{59.1}).

\begin{proposition}
\label{5.2}
Let $P_0$ be a domain, $U$ be a free $P_0$-module of rank four with dual module $U^*=\Hom_{P_0}(U,P_0)$, $x,y,z,w$ be a basis for $U$ with dual basis $x^*,y^*,z^*,w^*$ for $U^*$. Let $\phi_3\in D_3U^*$ have the form

\begin{equation}
\label{59.1} \phi_3=\begin{cases}
\phantom{+}x^{*(3)}+ax^*z^{*(2)}+dy^*z^{*(2)}+ey^*z^*w^*\\+fy^*w^{*(2)}
+gy^{*(2)}z^*+hy^{*(2)}w^*+iy^{*(3)}\\+jz^{*(3)}+kz^{*(2)}w^*+lz^*w^{*(2)}+mw^{*(3)},\end{cases}\end{equation}
for parameters $a,\dots,m$ in $P_0$.
Assume that 
\begin{enumerate}[\rm(a)]\item$a\neq 0$ in $P_0$ and
\item$\ell\phi_3\neq 0$ for all nonzero $\ell$ in $U$.
\end{enumerate}
 Then $\Gamma_{\phi_3}$ is not identically zero.
\end{proposition}

\begin{proof} 
The proof is by contradiction. We assume that $\Gamma_{\phi_3}$ is identically zero; we prove that $y\phi_3$ and $w\phi_3$ are linearly dependent. (Of course, this contradicts the hypothesis that $\ell\phi_3$ is nonzero whenever $\ell$ is a nonzero element of $U$).
We show that 
\begin{align*}
y\phi_3={}&gy^*z^*+hy^*w^*+iy^{*(2)}+dz^{*(2)}+ez^*w^*+fw^{*(2)}\quad\text{and}\\
w\phi_3={}&ey^*z^*+fy^*w^*+hy^{*(2)}+kz^{*(2)}+lz^*w^*+mw^{*(2)}\end{align*}
are linearly dependent by showing that the $2\times 2$ minors of \begin{equation}\label{matrix}\bmatrix g&h&i&d&e&f\\e&f&h&k&\ell&m\endbmatrix\end{equation} are all zero. 
The $2\times 2$ minors of (\ref{matrix}) are
\begingroup\allowdisplaybreaks
\begin{align*}
G_0={}& fg - eh,&
G_1={}& gh - ei,&
G_2={}& h^2  - fi,\\
G_3={}& gk- de, &
G_4={}& hk- df,&
G_5={}& ik- dh, \\
G_6={}& gl- e^2,&
G_7={}& hl- ef,&
G_8={}& il- eh,\\
G_9={}& dl- ek,&
G_{10}={}& gm- ef,&      
G_{11}={}& hm- f^2,\\
G_{12}={}& im- fh,&
G_{13}={}& dm- fk,\text{ and}&
G_{14}={}& em- fl.\\
\end{align*}\endgroup
Define $F_0,\dots,F_8$ to be the following elements of $P_0$:

\begingroup\allowdisplaybreaks
\begin{align*}
F_0={}&fi-h^2&
F_1={}&fg-2eh+il\\
F_2={}&gl-e^2,&
F_3={}&im-fh,\\
F_4={}&gm-2ef+hl,&
F_5={}&hm-f^2,\\
F_6={}&de^2-d^2f+fgj-2ehj+2dhk-ik^2-dgl+ijl,\\
F_7={}&gjl-e^2j+2dek-gk^2-d^2l,\text{ and}\\ 
F_8={}&gjm-2efj+e^2k+2dfk-hk^2+hjl-gkl-d^2m.
\end{align*}\endgroup
Let $\omega_U$ and $\omega_{U^*}$ represent the bases
$x\w y\w z\w w$ and  $x^*\w y^*\w z^*\w w^*$ of $\bw^4U$ and $\bw^4U^*$, respectively. In Lemma~\ref{l5.2} we show that
\begingroup\allowdisplaybreaks\begin{align}
\label{M26.1}\Gamma_{\phi_3}(x^{(2)}y^{(2)}\t \omega_U)={}&aF_0\omega_{U^*},\\
\label{M26.2}\Gamma_{\phi_3}(x^{(2)}yz\t \omega_U)={}&aF_1\omega_{U^*},\\\label{M26.3}\Gamma_{\phi_3}(x^{(2)}z^{(2)}\t \omega_U)={}&aF_2\omega_{U^*},\\
\label{M26.6}\Gamma_{\phi_3}(x^{(2)}yw)\t \omega_U)={}&aF_3\omega_{U^*},\\\label{M26.7}\Gamma_{\phi_3}(x^{(2)}zw\t \omega_U)={}&aF_4\omega_{U^*},\\
\label{M26.8}\Gamma_{\phi_3}(x^{(2)}w^{(2)}\t \omega_U)={}&aF_5\omega_{U^*},\\
\label{M26.4}\Gamma_{\phi_3}(xyz^{(2)}\t \omega_U)={}&F_6\omega_{U^*},\\
\label{M26.5}\Gamma_{\phi_3}(xz^{(3)}\t \omega_U)={}&F_7\omega_{U^*}, \text{ and}\\
\label{M26.9}\Gamma_{\phi_3}(xz^{(2)}w\t \omega_U)={}&F_8\omega_{U^*}.
\end{align}\endgroup
The element $a$ of the domain $P_0$ is nonzero. We have assumed that 
$\Gamma_{\phi_3}$ is identically zero. We conclude that 
$$F_0=F_1=F_2=F_3=F_4=F_5=F_6=F_7=F_8=0.$$ 
Straightforward calculation show that
\begingroup\allowdisplaybreaks
\begin{align*}
G_0^2={}&-glF_0+fgF_1-h^2F_2,&
G_1^2={}&-g^2F_0+giF_1-i^2F_2,\\
G_2={}&-F_0,&
G_3^2={}&-d^2F_2-g(F_7-jF_2),\\
G_4^2={}&-k^2F_0+f(-dF_2+jF_1-F_6),&
G_5^2={}&-d^2F_0+i(-dF_2+jF_1-F_6),\\ 
G_6={}&F_2,&
G_7^2={}&-f^2F_2+hlF_4-glF_5,\\
G_8^2={}&-glF_0+ilF_1-h^2F_2,&
G_9^2={}&-k^2F_2-l(F_7-jF_2),\\
G_{10}^2={}&-f^2F_2+gmF_4-glF_5,&
G_{11}={}&F_5,\\
G_{12}={}&F_3,&
G_{13}^2={}&-k^2F_5-m(F_8-jF_4+kF_2),\text{ and}\\
G_{14}^2={}&-m^2F_2+lm(F_4)-l^2F_5.\\
\end{align*}
\endgroup
It follows that $G_i=0$ for $0\le i\le 14$; all $2\times 2$ minors of 
(\ref{matrix}) are zero; the elements $y\phi_3$ and $w\phi_3$ of $U^*$ are linearly independent, and there exists a nonzero element $\ell$ of $\kk y\p\kk w\subseteq U$ with $\ell \phi_3=0$. This contradicts the ambient hypothesis. The proof is complete. \end{proof}
\begin{lemma}
\label{l5.2}
In the notation of Proposition~{\rm\ref{5.2}}, 
 the formulas
{\rm(\ref{M26.1})} to {\rm(\ref{M26.9})} all hold.
\end{lemma}

\begin{proof} 
We prove (\ref{M26.1}).
The expansion of $\Gamma_{\phi_3}(x^{(2)}y^{(2)}\t \omega_U)$ has six summands. Five of the summands have a factor of $xy\phi_3=0$ or $xw\phi_3=0$. Thus,
\begingroup\allowdisplaybreaks\begin{align*}\Gamma_{\phi_3}(x^{(2)}y^{(2)}\t \omega_U)={}&x^2\phi_3\w y^2\phi_3\w xz\phi_3\w yw\phi_3\\
{}={}&x^*\w (gz^*+hw^*+iy^*)\w az^*\w (ez^*+fw^*+hy^*)\\
{}={}&a(if-h^2)\omega_{U^*}=aF_0\omega_{U^*}. \end{align*}\endgroup

We prove  
(\ref{M26.2}). 
The expansion of $\Gamma_{\phi_3}(x^{(2)}yz\t \omega_U)$ has twelve summands. Ten of the   summands have a factor of $xy\phi_3=0$ or $xw\phi_3=0$. Thus,
\begingroup\allowdisplaybreaks\begin{align*}\Gamma_{\phi_3}(x^{(2)}yz\t \omega_U)={}&
x^2\phi_3\w y^2\phi_3\w xz\phi_3\w zw\phi_3+x^2\phi_3\w yz\phi_3\w xz\phi_3\w yw\phi_3\\
{}={}&\begin{cases}\phantom{+}
 x^*\w (gz^*+hw^*+iy^*)\w az^*\w (ey^*+kz^*+lw^*)\\
+x^*\w (dz^*+ew^*+gy^*)\w az^*\w (ez^*+fw^*+hy^*)
\end{cases}\\
{}={}&\Big(a(il-eh)+a(gf-eh)\Big)\omega_{U^*}\\
{}={}&a(fg-2eh+il)\omega_{U^*}=aF_1\omega_{U^*}. \end{align*}\endgroup

We prove
(\ref{M26.3}).
The expansion of $\Gamma_{\phi_3}(x^{(2)}z^{(2)}\t \omega_U)$ has six summands. Five  of the   summands have a factor of $xy\phi_3=0$ or $xw\phi_3=0$. Thus,
\begingroup\allowdisplaybreaks\begin{align*}\Gamma_{\phi_3}(x^{(2)}z^{(2)}\t \omega_U)={}&
x^2\phi_3\w yz\phi_3\w xz\phi_3\w zw\phi_3\\
{}={}&
 x^*\w (dz^*+ew^*+gy^*)\w az^*\w (ey^*+kz^*+lw^*)\\
{}={}&a(gl-e^2)\omega_{U^*}=aF_2\omega_{U^*}. \end{align*}\endgroup

We prove  
(\ref{M26.6}).
The expansion of $\Gamma_{\phi_3}(x^{(2)}yw\t \omega_U)$ has twelve summands. Eleven   of the   summands have a factor of $xy\phi_3=0$,  $xw\phi_3=0$, or $yw\phi_3\w yw\phi_3=0$. Thus,
\begingroup\allowdisplaybreaks\begin{align*}\Gamma_{\phi_3}(x^{(2)}yw\t \omega_U){}={}&
x^2\phi_3\w y^2\phi_3\w xz\phi_3\w w^2\phi_3\\
{}={}&
x^*\w (gz^*+hw^*+iy^*)\w az^*\w(fy^*+lz^*+mw^*)
\\
{}={}&a(im-fh)\omega_{U^*} 
= 
aF_3\omega_{U^*}. \end{align*}\endgroup

We prove 
(\ref{M26.7}).
The expansion of $\Gamma_{\phi_3}(x^{(2)}zw\t \omega_U)$ has twelve summands. Ten    of the   summands have a factor of $xy\phi_3=0$ or  $xw\phi_3=0$. Thus,
\begingroup\allowdisplaybreaks\begin{align*}\Gamma_{\phi_3}(x^{(2)}zw\t \omega_U)={}&
x^2\phi_3\w yz\phi_3\w xz\phi_3\w w^2\phi_3+x^2\phi_3\w yw\phi_3\w xz\phi_3\w zw\phi_3\\
{}={}&\begin{cases}\phantom{+}
x^*\w (dz^*+ew^*+gy^*)\w az^*\w(fy^*+lz^*+mw^*)\\
+x^*\w (ez^*+fw^*+hy^*)\w az^*\w(ey^*+kz^*+lw^*)\end{cases}\\
{}={}&\Big(a(gm-ef)+a(hl-ef)\Big)\omega_{U^*}\\
{}={}&aF_4\omega_{U^*}. \end{align*}\endgroup

We prove (\ref{M26.8}).
The expansion of $\Gamma_{\phi_3}(x^{(2)}w^{(2)}\t \omega_U)$ has six summands. Five     of the   summands have a factor of $xy\phi_3=0$ or  $xw\phi_3=0$. Thus,
\begingroup\allowdisplaybreaks\begin{align*}\Gamma_{\phi_3}(x^{(2)}w^{(2)}\t \omega_U)={}&
x^2\phi_3\w yw\phi_3\w xz\phi_3\w w^2\phi_3\\
{}={}&
x^*\w (ez^*+fw^*+hy^*)\w az^*\w(fy^*+lz^*+mw^*)\\
{}={}&a(hm-f^2)\omega_{U^*}\\
{}={}&aF_5\omega_{U^*}. \end{align*}\endgroup

We prove  
(\ref{M26.4}).
The expansion of $\Gamma_{\phi_3}(xyz^{(2)}\t \omega_U)$ has twelve summands. Ten  of the   summands have a factor of $xy\phi_3=0$,  $xw\phi_3=0$, $xz\phi_3\w xz\phi_3=0$, or $yz\phi_3\w yz\phi_3$. Thus,
\begingroup\allowdisplaybreaks\begin{align*}&\Gamma_{\phi_3}(xyz^{(2)}\t \omega_U)\\{}={}&
x^2\phi_3\w y^2\phi_3\w z^2\phi_3\w zw\phi_3+x^2\phi_3\w yz\phi_3\w z^2\phi_3\w yw\phi_3\\
{}={}&\begin{cases}\phantom{+}
x^*\w (gz^*+hw^*+iy^*)\w (ax^*+dy^*+jz^*+kw^*)\w (ey^*+kz^*+lw^*)
\\
 +x^*\w (dz^*+ew^*+gy^*)\w (ax^*+dy^*+jz^*+kw^*)\w (ez^*+fw^*+hy^*)\end{cases}\\
{}={}&\left(\det\bmatrix i&g&h\\d&j&k\\e&k&l\endbmatrix+\det \bmatrix g&d&e\\d&j&k\\h&e&f \endbmatrix\right)\omega_{U^*}\\
{}={}&(de^2-d^2f+fgj-2ehj+2dhk-ik^2-dgl+ijl)\omega_{U^*}\\
{}={}&F_6\omega_{U^*}. \end{align*}\endgroup

We prove  
(\ref{M26.5}).
The expansion of $\Gamma_{\phi_3}(xz^{(3)}\t \omega_U)$ has four summands. Three  of the   summands have a factor of $xy\phi_3=0$,  $xw\phi_3=0$, or $xz\phi_3\w xz\phi_3=0$. Thus,
\begingroup\allowdisplaybreaks\begin{align*}&\Gamma_{\phi_3}(xz^{(3)}\t \omega_U)\\{}={}&
x^2\phi_3\w yz\phi_3\w z^2\phi_3\w zw\phi_3\\
{}={}&
x^*\w (dz^*+ew^*+gy^*)\w (ax^*+dy^*+jz^*+kw^*)\w (ey^*+kz^*+lw^*)
\\
{}={}&\det\bmatrix g&d&e\\d&j&k\\e&k&l\endbmatrix\omega_{U^*}\\
{}={}&(-e^2j+2dek-gk^2-d^2l+gjl)\omega_{U^*}\\
{}={}&F_7\omega_{U^*}. \end{align*}\endgroup

We prove (\ref{M26.9}).
The expansion of $\Gamma_{\phi_3}(xz^{(2)}w\t \omega_U)$ has twelve summands. Ten   of the   summands have a factor of 
$$xy\phi_3=0,\quad  xw\phi_3=0,\quad zw\phi_3\w zw\phi_3=0,\quad\text{or}\quad xz\phi_3\w xz\phi_3=0.$$ Thus,
\begingroup\allowdisplaybreaks\begin{align*}&\Gamma_{\phi_3}(xz^{(2)}w\t \omega_U)\\
{}={}&
x^2\phi_3\w yw\phi_3\w z^2\phi_3\w zw\phi_3
+x^2\phi_3\w yz\phi_3\w z^2\phi_3\w w^2\phi_3
\\
{}={}&\begin{cases}\phantom{+}
 x^*\w (ez^*+fw^*+hy^*)\w (ax^*+dy^*+jz^*+kw^*)\w(ey^*+kz^*+lw^*)\\
+x^*\w (dz^*+ew^*+gy^*)\w (ax^*+dy^*+jz^*+kw^*)\w(fy^*+lz^*+mw^*)\end{cases}\\
{}={}&\left(
\det \bmatrix h&e&f\\d&j&k\\e&k&l\endbmatrix
+\det \bmatrix g&d&e\\d&j&k\\f&l&m\endbmatrix
\right)\omega_{U^*}\\
{}={}&
(-2efj+e^2k+2dfk-hk^2+hjl-gkl-d^2m+gjm)\omega_{U^*}\\
{}={}&F_8\omega_{U^*}. \end{align*}\endgroup
\vskip-24pt
\end{proof}

\begin{lemma}
\label{case3a} Lemma~{\rm\ref{ML}} holds when $\phi_3$ has the form of Lemma~{\rm\ref{61.1}.(\ref{61.1.a})} and $r=0$ or $r=3$.
\end{lemma}

\begin{proof}First we consider  the case $r=0$. In this case,
$\phi_3=x^{*(3)}+\phi_{3,0}$ with $\phi_{3,0}\in U_0^*$ and 
$\kk x^*\p U_0^*=U^*$. Furthermore, either the characteristic of $\kk$ is not two; or else, the  characteristic of $\kk$ is  two but there does not exist a basis $y^*$, $z^*$, $w^*$ for $U_0^*$ with
$\phi_{3,0}$ is  equal to $y^*z^*w^*$. 

Fix some basis $y^*,z^*,w^*$ for $U_0^*$. Let $x,y,z,w$ be the basis for $U$ which is dual to the basis $x^*,y^*,z^*,w^*$ for $U^*$.
Let \begingroup\allowdisplaybreaks
\begin{align*}&\omega_{U_0^*}\text{ be the basis }y^*\w z^*\w w^*\text{ of }\ts\bw^3U_0^*,\\ &\omega_{U_0}\text{ be the basis }y\w z\w w\text{ of }\ts\bw^3U_0,\\ 
&\omega_{U^*}\text{ be the basis }x^*\w y^*\w z^*\w w^*\text{ of }\ts\bw^4U^*,
\text{ and }\\
&\omega_{U}\text{ be the basis }x\w y\w z\w w\text{ of }\ts\bw^4U 
.\end{align*}\endgroup
According to Lemma~\ref{lemma3} 
 there is an element $\boxx\in D_3U_0$ so that $\Gamma_{\phi_{3,0}}(\boxx\t \omega_{U_0})=\omega_{U_0^*}$.
The fact that $x\phi_{3,0}=0$ and $\ell x^{*(3)}=0$ for $\ell\in U_0$ ensures that 
$$\Gamma_{\phi_3}(x\boxx\t \omega_{U})=x^2\phi_3\w \Gamma_{\phi_{3,0}\t \omega_{U_0}}(\boxx)=x^*\w \omega_{U_0^*}=\omega_{U^*}.$$
Conclude that $\Gamma_{\phi_3}\neq 0$.

Now we consider the case $r=3$. In this case,  
 $\phi_3=\a x^{*(3)}+x^*\phi_{2,0}+\phi_{3,0}$, with $\phi_{i,0}$ in $D_iU_0^*$ and $\ell_0\phi_{2,0}$ not zero  for $\ell_0\in U_0=\kk y\p\kk z\p \kk w$. It follows that $y\phi_{2,0}$, $z\phi_{2,0}$, and $w\phi_{2,0}$ are linearly independent in $U_0^*$. Hence, $y\phi_{2,0}\w z\phi_{2,0}\w w\phi_{2,0}$ is nonzero in $\bw^3U_0^*$ and
$$\Gamma_{\phi_3}(x^{4}\t \omega_U)=\a x^*\w y\phi_{2,0}\w z\phi_{2,0}\w w\phi_{2,0}$$ is nonzero in $\bw^4U^*$. \end{proof}

\section{The Macaulay inverse system does not have a cubic term, one case.}\label{nocubic,and}
In this section $\phi_3$ does not have any terms of the form $\ell^{*(3)}$ with $\ell^*\in U^*$, but $\phi_3$ can be written in the form
$$\phi_3=x^{*(2)}y^*+\phi_{2,0}x^*+\phi_{3,0},$$ with $\phi_{i,0}\in D_iU_0^*$,  
where  $U^*=\kk x^*\p U_0^*$ and $y^*$ is a nonzero element of $U_0^*$. 
(This case can be avoided in most characteristics, but is necessary in characteristic three.)
 In Propositions~\ref{6.1} and \ref{NYW} we prove that if $\ell\phi_3$ is nonzero for all nonzero $\ell$ in $U$, then $\Gamma_{\phi_3}$ is not identically zero (and, therefore, $A_{\phi_3}$ has the weak Lefschetz property by 
Lemma~\ref{28.10}.) 
The proof of Propositions~\ref{6.1} and \ref{NYW} proceed like the proof  of Proposition~\ref{5.2}: we assume that $\Gamma_{\phi_3}$ is identically zero and we exhibit a nonzero linear form $\ell$ in $U$ with $\ell\phi_3=0$. 

The hypothesis that $\Gamma_{\phi_3}$ is identically zero imposes further constraints on the form of $\phi_3$ that we identify before beginning the proof of Proposition~\ref{6.1}.

\begin{lemma}
\label{6.0}
Let $x^*$ and $y^*$ be linearly independent elements in the four dimensional vector space $U^*$ over the field $\kk$ and let
$$\phi_3=x^{*(2)}y^*+\phi_{2,0}x^*+\phi_{3,0}$$ be an element of $D_3U^*$ with $\phi_{i,0}\in D_iU_0^*$ for $U^*=\kk x^*\p U_0^*$ and $y^*\in U_0^*$. Assume that $\phi_{3,0}$ does not have any terms of the form $\ell^{*(3)}$ for $\ell^*\in U_0^*$. If $\Gamma_{\phi_3}$ is identically zero, then there exists a basis $y^*,z^*,w^*$ for $U_0^*$ such that
\begin{enumerate}
[\rm (i)]
\item
\label{6.0.i} $\phi_{2,0}$ is in the subspace of $U_0^*$ spanned by $z^{*(2)}$, $y^{*(2)}$, $y^*z^*$, and $y^*w^*$, and
\item
\label{6.0.ii} $\phi_{3,0}$ does not involve $z^*w^{*(2)}$.\end{enumerate}
\end{lemma}
\begin{remark} Assertion 
(\ref{6.0.i}) can also be written 
$\phi_{2,0}$ does not involve either  $z^*w^*$ or $w^{*(2)}$ and assertion (\ref{6.0.ii}) can also be written $\phi_{3,0}$ is in the subspace of $D_3U_0^*$ spanned by $y^{*(2)}z^*$, $y^{*(2)}w^*$, $y^*z^{*(2)}$, $y^*z^*w^*$, $y^*w^{*(2)}$,  and $z^{*(2)}w^*$.  \end{remark}

\begin{Cases} 
\label{Two Cases}
Once we prove Lemma~\ref{6.0}, then there are two cases.
\begin{enumerate}[\rm C{a}se 1.]
\item\label{Case1} The case where $\phi_{2,0}$ is in the subspace of $U_0^*$ spanned by $y^{*(2)}$, $y^*z^*$, and $y^*w^*$ is treated in Proposition~\ref{NYW}. 
\item\label{Case2} The case where $\phi_{2,0}$ also involves $z^{*(2)}$ is treated in Proposition~\ref{6.1}. 
\end{enumerate}
\end{Cases}

\begin{proof} 
We prove (\ref{6.0.i}). Let $y^*$, $Z^*$, $W^*$ be any basis for $U_0^*$. We are given $a$, $b$, $c$ in $\kk$ with 
$$\phi_3=x^{*(2)}y^*+ax^*Z^{*(2)}+bx^*Z^*W^*+cx^*W^{*(2)}+\phi_3',$$ with
$\phi_3'$ in the subspace of $D_3U_0^*$ spanned by $x^*y^{*(2)}$, $x^*y^*Z^*$, $x^*y^*W^*$,  $y^{*(2)}Z^*$, $y^{*(2)}W^*$, $y^*Z^{*(2)}$, $y^*Z^*W^*$, $y^*W^{*(2)}$,  $Z^{*(2)}W^*$,  and $Z^*W^{*(2)}$.
Let $\omega_U$ and $\omega_{U^*}$ be the bases $x\w y\w Z\w W$ and $x^*\w y^*\w Z^*\w W^*$ of $U$ and $U^*$, respectively.  Observe that
\begingroup\allowdisplaybreaks\begin{align*}\Gamma_{\phi_3}(x^{(4)}\t \omega_U)={}&x^2\phi_3\w xy\phi_3\w xZ\phi_3\w xW\phi_3\\
{}={}&y^*\w (x^*+\theta)\w (\a_1y^*+aZ^*+bW^*)\w (\a_2y^*+bZ^*+cW^*)\\
{}={}&(ac-b^2)\omega_{U^*},
\end{align*}
\endgroup
for some $\theta\in U_0^*$ and some $\a_1$ and $\a_2$ in $\kk$. The hypothesis that $\Gamma_{\phi_3}$ is identically zero guarantees that $$ac-b^2=0.$$ 
Apply Lemma~\ref{3.3} and choose a new basis $z^*$, $w^*$ for $\kk Z^*\p \kk W^*$ so that
$$ax^*Z^{*(2)}+bx^*Z^*W^*+cx^*W^{*(2)}\text{ is in }\kk x^*z^{*(2)}.$$ 
 We have established (\ref{6.0.i}).

We prove (\ref{6.0.ii}). At this point 
\begin{equation}\label{60.0.1}\phi_3= 
\begin{cases} \phantom{+}
  x^{*(2)}y^*
+d x^*y^{*(2)}+e y^{*(2)}z^*+fy^{*(2)}w^*\\
+g x^*z^{*(2)}+h y^*z^{*(2)}+iz^{*(2)}w^*
+k y^*w^{*(2)}+lz^*w^{*(2)}\\
+mx^*y^*z^*+nx^*y^*w^*+py^*z^*w^*,\end{cases}\end{equation}
for parameters $d,\dots,p$ from $\kk$.
Let $\omega_U$ and $\omega_{U^*}$ be the bases $x\w y\w z\w w$ and $x^*\w y^*\w z^*\w w^*$ of $U$ and $U^*$, respectively.
The expansion of 
$\Gamma_{\phi_3}(x^{*(2)}w^{*(2)}\t \omega_U)$ involves six summands. Five of the summands have a factor of $xw\phi_3\w xw\phi_3=0$, 
$$x^2\phi_3\w xw\phi_3\subseteq\ts \bw^2(\kk y^*)=0,\text{ or }
\bw^3\kk(x^2\phi_3,xz\phi_3,xw\phi_3,w^2\phi_3)\subseteq\ts \bw^3(\kk y^*\p\kk z^*)=0.$$ Thus,
 \begin{align*}
&\Gamma_{\phi_3}(x^{*(2)}w^{*(2)}\t \omega_U)\\{}={}&x^2\phi_3\w xy\phi_3\w zw\phi_3\w w^2\phi_3\\{}={}&(y^*\w(x^*+dy^*+mz^*+nw^*)\w(iz^*+lw^*+py^*)\w(ky^*+lz^*)
\\{}={}&l^2\omega_{U^*}.\end{align*}
The hypothesis that $\Gamma_{\phi_3}$ is identically zero guarantees that $l=0$; and this completes the proof of (\ref{6.0.ii}).
\end{proof}

We treat Case~\ref{Case2} of \ref{Two Cases}. The Macaulay inverse system for this case is the $\phi_3$ of (\ref{60.0.1}) with $l$ set equal to zero. We have recorded this Macaulay inverse system as (\ref{60.2.1}).
\begin{proposition}
\label{6.1}
Let $P_0$ be a domain, $U$ be a free $P_0$-module of rank four with dual module $U^*=\Hom_{P_0}(U,P_0)$, $x,y,z,w$ be a basis for $U$ with dual basis $x^*,y^*,z^*,w^*$ for $U^*$. Let $\phi_3\in D_3U^*$ have the form
\begin{equation}\label{60.2.1}\phi_3= 
\begin{cases} \phantom{+}
 a x^{*(2)}y^*
+d x^*y^{*(2)}+e y^{*(2)}z^*+fy^{*(2)}w^*\\
+g x^*z^{*(2)}
+h y^*z^{*(2)}+iz^{*(2)}w^*
+k y^*w^{*(2)}\\
+mx^*y^*z^*+nx^*y^*w^*+py^*z^*w^*,\end{cases}\end{equation}
for parameters $a,\dots,p$ from $P_0$. Assume that 
\begin{enumerate}[\rm(a)]
\item$a\neq 0$ in $P_0$,
\item$g\neq 0$ in $P_0$, 
and 
\item$\ell\phi_3\neq 0$ for all nonzero $\ell$ in $U$. 
\end{enumerate}
 Then $\Gamma_{\phi_3}$ is not identically zero.
\end{proposition} 

\begin{proof} 
The proof is by contradiction. We assume that $\Gamma_{\phi_3}$ is identically zero; we prove that the elements 
\begin{align*}x\phi_3={}&ax^*y^*+dy^{*(2)}+my^*z^*+ny^*w^*+gz^{*(2)}\text{ and}\\
w\phi_3={}&nx^*y^*+fy^{*(2)}+py^*z^*+ky^*w^*+iz^{*(2)}\end{align*}
of $U^*$ are linearly dependent.
We do this by showing that the $2\times 2$ minors of
the matrix
\begin{equation}\label{matrix2}\bmatrix 
a&d&m&n&g\\
n&f&p&k&i\endbmatrix\end{equation} are all zero.

In Lemma~\ref{60.2} we calculate $\Gamma_{\phi_3}(-\t \omega_U)$ for five elements of $D_3U$. The assumption that $\Gamma_{\phi_3}$ is identically zero, together with the hypothesis that $ag\neq 0$ guarantees that 
$n^2-ak$, $ai-gn$, $im-gp$, $gk-in$, 
and \begin{equation} \label{I0}
I_0=- af^2g - d^2 gk - adhk + dkm^2  
+ 2dfgn + dhn^2  - 2dmnp + adp^2\end{equation}
all are zero. 
In Lemma~\ref{60.D} we show that $dk-fn$ is also zero.

The entries $a$ and $g$ of matrix~\ref{matrix2} are nonzero. We know that columns 1 and 4
 are linearly dependent, as are  columns 1 and 5, columns 3 and 5,  columns  4 and 5,  and columns 2 and 4.
If $n\neq 0$, then column 4 is a non-zero multiple of column 1 and therefore columns 1 and 2 are linearly dependent, indeed, in this case, column 1  is a basis for the column space of the matrix (\ref{matrix2}). 

On the other hand, if $n=0$, then the fact that columns 1,3,4,5 are all multiples of column 1 and $n$ is the bottom entry of column 1 forces $n=p=i=k=0$. When   $n=p=i=k=0$, then $\Gamma_{\phi_3}(xy^{(3)}\t \omega_U)=I_0\omega_{U^*}$ (from Lemma \ref{60.2}) becomes $-af^2g=0$. The parameters  $a$ and $g$ are nonzero; hence $f=0$ and once again 
the matrix (\ref{matrix2}) has rank one.

Thus, in every case,  $x\phi_3$ and $w\phi_3$ are linearly dependent and  there is a nonzero linear form $\ell=\a x+\beta w$ with $\ell \phi_3=0$, which is a contradiction.  
\end{proof}

\begin{lemma}
 Retain the notation of Proposition~{\rm\ref{6.1}}. Let $\omega_U$ and $\omega_{U^*}$ be the basis elements $x\w y\w z\w w$ and $x^*\w y^*\w z^*\w
 w^*$ of $\bw^4U$ and $\bw^4U^*$, respectively. \label{60.2} Then the following statements hold{\rm :}
 \begin{enumerate}[\rm(a)]
\item\label{6.4a}$\Gamma_{\phi_3}(x^{(3)}y\t \omega_U)=ag(n^2-ak)\omega_{U^*}$,
\item \label{6.4b}$\Gamma_{\phi_3}(x^{(2)}z^{(2)}\t \omega_U)=(ai-gn)^2\omega_{U^*}$,
\item\label{6.4c}$\Gamma_{\phi_3}(z^{(4)}\t \omega_U)=(im-gp)^2\omega_{U^*}$,
\item\label{6.4d}  $\Gamma_{\phi_3}(z^{(2)}w^{(2)}\t \omega_U)=(gk-in)^2\omega_{U^*}$,  and
\item\label{6.4e} $\Gamma_{\phi_3}(xy^{(3)}\t \omega_U)=I_0 
\omega_{U^*}$, where $I_0$ is given in {\rm(\ref{I0})}.
\end{enumerate} 
\end{lemma}

\begin{proof} 
(\ref{6.4a}) The expansion of $\Gamma_{\phi_3}(x^{(3)}y\t \omega_U)$ has four summands. Three of the summands have a factor of  $xy\phi_3\w xy\phi_3=0$ 
or $x^2\phi_3\w xw\phi_3\in \bw^2P_0 y^*=0$. It follows that
\begingroup\allowdisplaybreaks\begin{align*}
&\Gamma_{\phi_3}(x^{(3)}y\t \omega_U)\\
{}={}&
 x^2\phi_3\w xy\phi_3\w xz\phi_3\w yw\phi_3
\\
{}={}&ay^*\w (ax^*+dy^*+mz^*+nw^*)\w (gz^*+my^*)\w (fy^*+kw^*+nx^*+pz^*)\\
{}={}& -ag(ak-n^2)\omega_{U^*}.
\end{align*}\endgroup

\ms\noindent (\ref{6.4b}) The expansion of $\Gamma_{\phi_3}(x^{(2)}z^{(2)}\t \omega_U)$ has six summands. Four of the summands contain a factor of
$xz\phi_3\w xz\phi_3=0$, or $x^2\phi_3\w xw\phi_3\in \bw^2\kk y^*=0$, or
$$x^2\phi_3\w xz\phi_3\w zw\phi_3\in \ts\bw^3(\kk y^*\p \kk z^*)=0.$$
It follows that \begingroup\allowdisplaybreaks\begin{align*}
&
\Gamma_{\phi_3}(x^{(2)}z^{(2)}\t \omega_U)\\
{}={}&\begin{cases}\phantom{+}
 x^2\phi_3\w xy\phi_3\w z^2\phi_3\w zw\phi_3\\
+xz\phi_3\w  xy\phi_3\w z^2\phi_3\w xw\phi_3\\
\end{cases}
  \\ 
{}={}&\begin{cases}\phantom{+}
 ay^*\w (ax^*+dy^*+mz^*+nw^*)\w (gx^*+hy^*+iw^*)\w (iz^*+py^*)\\
+(gz^*+my^*)\w  (ax^*+dy^*+mz^*+nw^*)\w  (gx^*+hy^*+iw^*)\w  ny^*\\
\end{cases}\\
{}={}&\Big(ai(ai-ng)-gn(ai-gn)\Big)\omega_{U^*}=(ai-ng)^2\omega_{U^*}.
\end{align*}\endgroup

\ms\noindent (\ref{6.4c}) One computes that
\begingroup\allowdisplaybreaks\begin{align*}
&
\Gamma_{\phi_3}(z^{(4)}\t \omega_U)\\
{}={}&xz\phi_3\w yz\phi_3\w z^2\phi_3\w zw\phi_3\\
{}={}&(gz^*+my^*)\w (ey^*+hz^*+mx^*+pw^*)\w (gx^*+hy^*+iw^*)\w (iz^*+py^*)\\
{}={}&(mi-gp)^2\omega_{U^*}.\end{align*}\endgroup

\ms\noindent (\ref{6.4d}) The expansion of $\Gamma_{\phi_3}(z^{(2)}w^{(2)}\t \omega_U)$ has six summands. Four of the summands have a factor of $zw\phi_3\w zw\phi_3=0$, or
 $$\ts xz\phi_3\w zw\phi_3\w w^2\phi_3\in \bw^3(\kk y^*\p \kk z^*)=0, \quad\text{ or }\quad xw\phi_3\w w^2\phi_3\in \bw^2\kk y^*=0.$$ It follows that 
\begingroup\allowdisplaybreaks\begin{align*}
&
\Gamma_{\phi_3}(z^{(2)}w^{(2)}\t \omega_U)\\
{}={}&xz\phi_3\w yw\phi_3\w z^2\phi_3\w w^2\phi_3+xw\phi_3\w yw\phi_3\w z^2\phi_3\w zw\phi_3\\
{}={}&\begin{cases}\phantom{+}
 (gz^*+my^*)\w (fy^*+kw^*+nx^*+pz^*)\w (gx^*+hy^*+iw^*)\w ky^*\\
+ny^*\w (fy^*+kw^*+nx^*+pz^*)\w (gx^*+hy^*+iw^*)\w (iz^*+py^*)\\
\end{cases}
\\
{}={}&
 -kg(ni-gk)\omega_{U^*}
+ni(ni-gk)\omega_{U^*}
\\
{}={}&(ni-gk)^2\omega_{U^*}.
\end{align*}\endgroup

\ms\noindent (\ref{6.4e}) The expansion of $\Gamma_{\phi_3}(xy^{(3)}\t \omega_U)$ has four summands. One of the summands has a factor of $xy\phi_3\w xy\phi_3=0$. Thus,
\begingroup\allowdisplaybreaks\begin{align*}&\Gamma_{\phi_3}(xy^{(3)}\t \omega_U)\\{}={}& \begin{cases}\phantom{+}
 x^2\phi_3\w y^2\phi_3 \w yz\phi_3\w yw\phi_3\\
+xy\phi_3\w y^2\phi_3 \w xz\phi_3\w yw\phi_3\\
+xy\phi_3\w y^2\phi_3 \w yz\phi_3\w xw\phi_3\end{cases}\\
{}={}& \begin{cases}\phantom{+}
 \begin{cases}ay^*\w (dx^*+ez^*+fw^*)  \w (ey^*+hz^*+mx^*+pw^*)\\\w (fy^*+kw^*+nx^*+pz^*)\end{cases}\\
+\begin{cases}(ax^*+dy^*+mz^*+nw^*)\w (dx^*+ez^*+fw^*) \w (gz^*+my^*)\\
\w (fy^*+kw^*+nx^*+pz^*)\end{cases}\\
+\begin{cases}(ax^*+dy^*+mz^*+nw^*)\w (dx^*+ez^*+fw^*)\\ \w (ey^*+hz^*+mx^*+pw^*)\w ny^*\end{cases}\end{cases}\\
{}={}&\left( 
-a\det\bmatrix d&e&f\\m&h&p\\n&p&k\endbmatrix 
+\det\bmatrix a&d&m&n\\d&0&e&f\\0&m&g&0\\n&f&p&k\endbmatrix
+n\det\bmatrix a&m&n\\d&e&f\\m&h&p\endbmatrix\right)\omega_{U^*}\\
{}={}&(-af^2g-d^2gk-adhk+dkm^2+2dfgn+dhn^2-2dmnp+adp^2)\omega_{U^*}
\\{}={}& I_0\omega_{U^*}.\end{align*}\endgroup
\vskip-24pt\end{proof}

\begin{lemma}
\label{60.D} The parameters of Proposition~{\rm\ref{6.1}} satisfy
$$ag^2(dk-fn)^2\in (n^2-ak,ai-gn,im-gp).$$\vskip-24pt\end{lemma} 

\begin{proof}
Recall that $I_0$, from (\ref{I0}), is equal to zero. 
A straightforward calculation shows that
\begingroup\allowdisplaybreaks
\begin{align*}
&-gn^2 I_0\\
&+   (-d^2g^2k+2dfg^2n-dim^2n+dghn^2)(n^2-ak)\\ 
& + (-dkm^2n+dmn^2p)(ai-gn)\\
& + (dmn^3-adn^2p)(im-gp)
\end{align*}\endgroup
is equal to $ag^2(dk-fn)^2$.
\end{proof}

\section{The Macaulay inverse system does not have a cubic term, the other case.}\label{nocubic,other}
We treat Case~\ref{Case1} from \ref{Two Cases}. The statement is the same as  the statement of Proposition~\ref{6.1}, except now ``$g$'' is equal to zero.

\begin{proposition}
\label{NYW}
Let $P_0$ be a domain, $U$ be a free $P_0$-module of rank four with dual module $U^*=\Hom_{P_0}(U,P_0)$, $x,y,z,w$ be a basis for $U$ with dual basis $x^*,y^*,z^*,w^*$ for $U^*$. Let $\phi_3\in D_3U^*$ have the form
\begin{equation}\label{7.2.1}\phi_3= 
\begin{cases} \phantom{+}
 a x^{*(2)}y^*
+d x^*y^{*(2)}+e y^{*(2)}z^*+fy^{*(2)}w^*\\
+h y^*z^{*(2)}+iz^{*(2)}w^*
+k y^*w^{*(2)}\\
+mx^*y^*z^*+nx^*y^*w^*+py^*z^*w^*,\end{cases}\end{equation}
for parameters $a,\dots,p$ from $P_0$. Assume that 
\begin{enumerate}[\rm(a)]
\item$a\neq 0$ in $P_0$, and 
\item$\ell\phi_3\neq 0$ for all nonzero $\ell$ in $U$. 
\end{enumerate}
 Then $\Gamma_{\phi_3}$ is not identically zero.
\end{proposition} 

\begin{proof} 
The proof is by contradiction. We suppose that $\Gamma_{\phi_3}$
is identically zero. We prove that there is a nonzero element $\ell$ of $U$ with $\ell\phi_3=0$. 

We first show that $i=0$. Fix the bases $$\omega_U=x\w y\w z\w w\quad\text{and}\quad \omega_{U^*}=x^*\w y^*\w z^*\w w^*$$   of $\bw^4U$ and $\bw^4U^*$, respectively. The expansion of $\Gamma_{\phi_3}(x^{(2)}z^{(2)}\t \omega_U)$ has six summands; however five of the summands have a factor of $xz\phi_z\w xz\phi_3=0$ or a factor from $$\ts\bw^2(\kk x^2\phi_3+\kk xz\phi_3+ \kk xw\phi_3)\subseteq \bw^2\kk y^*=0.$$ Thus, \begingroup\allowdisplaybreaks
\begin{align*}&\Gamma_{\phi_3}(x^{(2)}z^{(2)}\t \omega_U)\\
{}={}&
x^2\phi_3\w xy\phi_3\w  z^2\phi_3\w  zw \phi_3\\
{}={}&(ay^*)\w (ax^*+dy^*+mz^*+nw^*)\w  (hy^*+iw^*)\w  (iz^*+py^*)\\
{}={}&a^2i^2\omega_{U^*}.\end{align*}\endgroup
The parameter $a$ is not zero. We have assumed that $\Gamma_{\phi_3}$ is identically zero. We conclude that $i=0$.

Ultimately, we prove that the elements 
\begin{align*}
x\phi_3={}& ax^*y^*+dy^{*(2)}+my^*z^*+ny^*w^*\\
z\phi_3={}&mx^*y^*+ey^{*(2)}+hy^*z^*+py^*w^*\\
w\phi_3={}&nx^*y^*+fy^{*(2)}+py^*z^*+ky^*w^*\end{align*}
of $D_2U^*$ are linearly dependent. We do this by showing that the maximal minors of the matrix 
$$\bmatrix
a&d&m&n\\
m&e&h&p\\
n&f&p&k\endbmatrix$$
are zero. We view the above matrix as the submatrix of the symmetric matrix
\begin{equation}\label{M}M=\bmatrix
a&d&m&n\\
d&0&e&f\\
m&e&h&p\\
n&f&p&k\endbmatrix,\end{equation} which is obtained by deleting row 2. 
In Lemma~\ref{7.2} we prove that
$$\Gamma_{\phi_3}(x^{(2)}y^{(2)}\t\omega_U)=-a\det M[2;2]\omega_{U^*}\quad \text{and}\quad
\Gamma_{\phi_3}(y^{(4)}\t\omega_U)=\det M\omega_{U^*},$$
where \begin{equation}\begin{array}{l}\text{$M[r_1,\dots, r_s;c_1,\dots,c_t]$ is the submatrix of $M$ obtained by deleting}\\\text{rows $r_1,\allowbreak\dots,\allowbreak  r_s$ and columns $c_1,\dots,c_t$.}\end{array}\label{M[}\end{equation} 

We have assumed that $\Gamma_{\phi_3}$ is identically zero and that $a$ is nonzero. We conclude that $M[2;2]$ and $M$ both have determinant zero. Observe that \begin{align}
\label{7.1.3}(\det M[2;1])^2={}& \det M[1;1]\det M[2;2]-\det M[1,2;1,2]\det M,\\
(\det M[2;3])^2={}& \det M[3;3]\det M[2;2]-\det M[2,3;2,3]\det M,\text{ and}\notag\\
(\det M[2;4])^2={}&\det M[4;4]\det M[2;2]-\det M[2,4;2,4]\det M.\notag\end{align}
One can check these formulas by hand or  one can use the characteristic free straightening technique of  
\cite{DEP} to verify these formulas; see Remark~\ref{DEP}. At any rate,
all four maximal minors of $M$ with row 2 deleted are zero;  $x\phi_3$, $z\phi_3$, $w\phi_3$ are  linearly dependent elements of $D_2U^*$; and $\ell\phi_3=0$ for some nonzero element $\ell$ of $\kk x\p\kk z\p\kk w$. This contradiction completes the proof. 
 \end{proof}

\begin{lemma}
\label{7.2} Let $\phi_3$ be the element of $D_3U^*$ which is given in {\rm(\ref{7.2.1})} with $i$  equal to zero and $M$ be the matrix of {\rm(\ref{M})}. Adopt the convention of {(\rm\ref{M[})}. Then 
\begin{enumerate}
[\rm (a)]
\item $\Gamma_{\phi_3}(x^{(2)}y^{(2)}\t\omega_U)=-a(\det M[2;2])\,\omega_{U^*}$ and
\item $\Gamma_{\phi_3}(y^{(4)}\t\omega_U)=(\det M)\,\omega_{U^*}$.
\end{enumerate}
\end{lemma}

\begin{proof} 
The expansion of  $\Gamma_{\phi_3}(x^{(2)}y^{(2)}\t\omega_U)$ consists of six summands; however five of the summands have a factor of $xy\phi_3\w xy\phi_3=0$ or have a factor from
$$\ts\bw^2\kk(x^2\phi_3,\,xz\phi_3,\,xw\phi_3)\subseteq \bw^2\kk y^*=0.$$ 
Thus,
\begingroup\allowdisplaybreaks
\begin{align*}
\Gamma_{\phi_3}(x^{(2)}y^{(2)}\t \omega_U)
{}={}&
x^2\phi_3\w xy\phi_3\w yz\phi_3\w yw\phi_3\\
{}={}&\begin{cases}
(ay^*)\w (ax^*+dy^*+mz^*+nw^*)\\\w (ey^*+hz^*+mx^*+pw^*)\w (fy^*+kw^*+nx^*+pz^*)
\end{cases}
\\
{}={}&-a\det \bmatrix a&m&n\\m&h&p\\n&p&k\endbmatrix.  
\intertext{We also compute}
\Gamma_{\phi_3}(y^{(4)}\t\omega_U)
{}={}&xy\phi_3\w y^2\phi_3\w yz\phi_3\w yw\phi_3\\
{}={}&\begin{cases}(ax^*+dy^*+mz^*+nw^*)\w (dw^*+ez^*+fw^*)\\\w (ey^*+hz^*+mx^*+pw^*)\w(fy^*+kw^*+nx^*+pz^*)\end{cases}\\
{}={}&\det\bmatrix a&d&m&n\\d&0&e&f\\m&e&h&p\\n&f&p&k\endbmatrix\omega_{U^*}.\end{align*}\endgroup
\end{proof}

\begin{remark}\label{DEP}
The identities of (\ref{7.1.3}) about the minors of the $4\times 4$ symmetric matrix $M$ of (\ref{M}) actually hold in a general situation. The matrix $M$ need not be symmetric and the size of $M$ need not be $4$. 

Recall the Pl\"ucker relations on the maximal minors of a matrix. 
Let $Y$ be an $r\times c$  matrix, with $r\le c$. The entries of $Y$ are in the commutative ring $R$. Let $a_1,\dots,a_{r-1}$ and $b_1,\dots,b_{r+1}$  be integers between $1$ and $c$ and $Y(s_1,\dots, s_\ell)$ represents the matrix
whose columns are the columns  $s_1,s_2,\dots,s_\ell$  of $Y$.
 Then
\begin{equation}\label{61.2.3}\sum_{i=1}^{r+1}\det Y(a_1,\dots,a_{r-1},b_i)\det Y(b_1,\dots,\widehat{b_i},\dots b_{r+1})=0,\end{equation}(where $\widehat{b_i}$ means that column $b_i$ has been deleted.)
Let $M$ be a $4\times 4$ matrix and $Y$ be the matrix
$$Y=[\begin{matrix} M|J\end{matrix}],$$ where $$J=\bmatrix
0&0&0&1\\
0&0&1&0\\
0&1&0&0\\
1&0&0&0\endbmatrix.$$ 
Apply (\ref{61.2.3}) to the matrix $Y=[M|J]$, where $M$ is given  in (\ref{M}). Take $$(\{a_1,a_1,a_3\},
\{b_1,b_2,b_3,b_4,b_5\})$$ to be
$$(\{2,3,4\},\{1,3,4,7,8\}),\ \
 (\{1,2,4\},\{1,3,4,6,7\}), \ \ \text{and}\ \  (\{1,2,3\},\{1,3,4,5,7\})$$ to obtain 
(\ref{7.1.3}).
\end{remark}

\section{The three variables Theorem.}\label{3var}
In this section we state and prove the three variable version of the Main Theorem (Theorem \ref{main}). Our precise formulation of the three variable version (see Lemma~\ref{lemma3}) is used in the inductive part of the proof of Theorem~\ref{main}. (See the case $r=0$ in Lemma~\ref{case3a}.)
Furthermore, we prove the three variable version using the same argument as we use for the four variable version; except there are fewer cases and each calculation is more straightforward. The reader might want to read the present section as a preparation for reading the proof of Theorem~\ref{main}.

\begin{theorem}
\label{main-3}Let $\kk$ be a field 
 and $A$ be a standard graded Artinian Gorenstein $\kk$-algebra of embedding dimension three and socle degree three.
If the characteristic of $\kk$ is different than two, then $A$ has the weak Lefschetz property. 
If the characteristic of $A$ is equal to two, then  
$A$ has the weak Lefschetz property
if and only if $A$ is not isomorphic to
\begin{equation}\label{1.1.1-3}\frac{\kk[x,y,z]}{(x^2,y^2,z^2)}.\end{equation}
\end{theorem}
\begin{proof}
It is clear that the ring  $A$ of (\ref{1.1.1-3}) does not satisfy the weak Lefschetz property. Indeed, if $\ell$ is a nonzero linear form of $A$, then $\ell^2=0$ in $A$; so, in particular, $\ell:A_1\to A_2$ is not injective.
The Macaulay inverse system for $A$ is $\phi_3=x^*y^*z^*$.

To complete the proof of Theorem~\ref{main-3},  we adopt Data~\ref{28.4}, with $d=3$, and use the method of the proof of Theorem~\ref{main}. Let $U=A_1$ and $\phi_3\in D_3U^*$ be a Macaulay inverse system for $A$. Assume that $\ell \phi_3$ is nonzero for all $\ell$ in $U$. Assume also that either the characteristic of $\kk$ is not two or that the characteristic of $\kk$ equals two but there does not equal a basis $x^*,y^*,z^*$ for $U^*$ with $\phi_3=x^*y^*z^*$. In Lemma~\ref{lemma3}, we prove that $\Gamma_{\phi_3}\neq 0$. Apply
Lemma~\ref{28.10} to complete the proof.   
\end{proof}

\begin{lemma}
\label{lemma3} Let $\kk$ be a field, $U$ be a three-dimensional vector space over $\kk$, and $\phi_3$ be an element of $D_3U^*$. Assume
\begin{enumerate}[\rm(a)]
\item  either the characteristic of $\kk$ is not two, or   the characteristic of $\kk$ is equal to two but $\phi_3\neq x^*y^*z^*$for any basis $x^*$, $y^*$, $z^*$ of $U^*$; and
\item $\ell\phi_3$ is nonzero for every nonzero $\ell$ in $U$.
\end{enumerate}
Then $\Gamma_{\phi_3}$ is not identically zero.
\end{lemma}

\begin{proof}
The proof is by contradiction. We assume that $\Gamma_{\phi_3}$ is identically zero. 
There are two cases. Either $\phi_3$ has the form of (\ref{61.1.a}) or (\ref{61.1.b}) from Lemma~\ref{61.1}.

We first assume that $\phi_3$ has the form of Lemma~\ref{61.1}.(\ref{61.1.a}).
In other words, there is a basis $x^*,y^*,z^*$ for $U^*$ so that  $\phi_3=ax^{*(3)}+x^*\phi_{2,0}+\phi_{3,0}$ with $a\in \kk$, $a\neq 0$, and  $\phi_{i,0}\in D_i\kk(y^*,z^*)$. Thus, 
 there are parameters $a,b,\dots,j$, in $\kk$,  with $a\neq 0$, such that
\begin{equation}\label{8.1.2}\phi_3=\begin{cases}
\phantom{+}ax^{*(3)}+dx^*y^{*(2)}+ex^*y^*z^*+fx^*z^{*(2)}\\+gy^{*(3)}+hy^{*(2)}z^*+iy^*z^{*(2)}+jz^{*(3)}.\end{cases}
\end{equation}
Let $x,y,z$ be the basis for $U$ which is dual to the basis $x^*,y^*,z^*$ for $U^*$. Let $\omega_U$ and $\omega_{U^*}$ be the bases $x\w y\w z$ and $x^*\w y^*\w z^*$ for $\bw^3U$ and $\bw^3 U^*$, respectively.

  Observe that \begin{align*}
\Gamma_{\phi_3}(x^{(3)}\t \omega_U)={}&x^2\phi_3\w xy\phi_3\w xz\phi_3\\
{}={}&ax^2\w(dy^*+ez^*)\w(ey^*+fz^*)\\
{}={}&a(df-e^2)\omega_{U^*}.\end{align*}
The assumption that $\Gamma_{\phi_3}$ is identically zero, combined with the hypothesis that $a\neq 0$ yields that $df-e^2=0$.

Apply Lemma~\ref{3.3} and change the basis for $\kk y^*\p \kk z^*$ in order to
write 
$$dx^*y^{*(2)}+ex^*y^*z^*+fx^*z^{*(2)}$$ in the form  $dx^*Y^{*(2)}$ with $e$ and $f$ equal to zero for some new basis $Y^*,Z^*$ for  $\kk y^*\p \kk z^*$.  

In Lemma~\ref{8.3} we show that when $\phi_3$ is given in (\ref{8.1.2}) with $e=f=0$, then 
\begingroup\allowdisplaybreaks\begin{align*}
\Gamma_{\phi_3}(y^{(3)}\t \omega_U)={}&-d^2i\omega_{U^*},\\
\Gamma_{\phi_3}(x^{(2)}z\t \omega_U)={}&adj\omega_{U^*}, \text{ and}\\
\Gamma_{\phi_3}(xy^{(2)}\t \omega_U)={}&a(gi-h^2)\omega_{U^*}.\end{align*}\endgroup
Recall our assumption that $\Gamma_{\phi_3}$ is identically zero and our local hypothesis that $a\neq 0$. If $d\neq0$, then $i=j=h=0$, 
$$\phi_3=
ax^{*(3)}
+dx^*y^{*(2)}
+gy^{*(3)},
$$
(because  $e$ and $f$  continue to be zero), and  
$z\phi_3=0$, which is a contradiction. Thus, $d=0$.

In Lemma~\ref{8.3}, we show that
\begingroup\allowdisplaybreaks
\begin{align*} 
\Gamma_{\phi_3}(xy^{(2)}\t \omega_U)={}&a(gi-h^2)\omega_{U^*}, \\ 
\Gamma_{\phi_3}(xz^{(2)}\t \omega_U)={}&a(hj-i^2)\omega_{U^*}, \text{ and}\\ 
\Gamma_{\phi_3}(xyz\t \omega_U)={}&a(gj-hi)\omega_{U^*}. \end{align*}\endgroup
Continue to assume that $\Gamma_{\phi_3}=0$, 
 $a\neq 0$, but $d=e=f=0$. It follows that  
 $h^2=gi$, $hi=gj$, $hj=i^2$, and 
$$\bmatrix g&h&i\\h&i&j\endbmatrix$$ has rank at most one.
Hence some nonzero linear combination of 
$$y\phi_3=gy^{*(2)}+hy^*z^*+iz^{*(2)}\quad\text{and}\quad z\phi_3=hy^{*(2)}+iy^*z^*+jz^{*(2)}$$
 is zero, which is a contradiction. This contradiction establishes Lemma~\ref{lemma3} for $\phi_3$ described in Lemma~\ref{61.1}.(\ref{61.1.a}).

\bs Now we assume that $\phi_3$ has the form of Lemma~\ref{61.1}.(\ref{61.1.b}). In other words, there is a basis $x^*$, $y^*$, $z^*$ for $U^*$ so that 
$\phi_3= bx^{*(2)}y^*+x\phi_{2,0}+\phi_{3,0}$ with $\phi_{i,0}$ in $D_i\kk(y^*,z^*)$,  $b\in\kk$, and $b\neq 0$. Thus,   there are parameters $a,b,\dots,j$, with $b\neq 0$,  such that
\begin{equation}\label{8.2.2}\phi_3=\begin{cases}
\phantom{+}bx^{*(2)}y^*+dx^*y^{*(2)}+ex^*y^*z^*+fx^*z^{*(2)}\\+gy^{*(3)}+hy^{*(2)}z^*+iy^*z^{*(2)}+jz^{*(3)}.\end{cases}
\end{equation}
It is shown in Lemma~\ref{8.4} 
 that $$\Gamma_{\phi_3}(x^{(3)}\t \omega_{U})=-b^2f \omega_{U^*}\quad\text{and}\quad \Gamma_{\phi_3}(x^{(2)}z\t \omega_{U})=-b^2j \omega_{U^*}.$$ Combine the assumption that $\Gamma_{\phi_3}$ is identically zero with the local hypothesis that $b\neq 0$ in order to conclude that $f=j=0$.
It  is shown in Lemma~\ref{8.5}  
that when $f=0$, then 
$$\Gamma_{\phi_3}(x^{(2)}y\t \omega_U)=b(e^2-bi)\omega_{U^*}\quad\text{and}\quad
\Gamma_{\phi_3}(y^{(3)}\t \omega_U)=(\det M) \,\omega_{U^*},$$ where 
\begin{equation}\label{M2}M=\bmatrix 
b&d&e\\d&g&h\\e&h&i\endbmatrix.\end{equation}
The assumption that $\Gamma_{\phi_3}$ is identically zero, together with the local hypothesis that $b\neq 0$, yields that 
 $e^2-bi=0$ and $\det M=0$.
It follows that all three maximal minors of 
$M$ with row two deleted are zero. This conclusion follows from the technique of (\ref{61.2.3}) applied to the matrix
$$
\bmatrix 
b&d&e&0&0&1\\d&g&h&0&1&0\\e&h&i&1&0&0\endbmatrix.$$
Take columns $\{2,3\}$ to be $\{a_1,\dots,a_{r-1}\}$ and columns $\{1,3,5,6\}$
to be $\{b_1,\dots,\allowbreak b_{r+1}\}$ to see that
$$\det \bmatrix d&e\\h&i\endbmatrix 
\det \bmatrix d&h\\e&i\endbmatrix+ i\det M -\det \bmatrix g&h\\h&i\endbmatrix 
\det \bmatrix b&e\\e&i\endbmatrix=0.$$ The expressions $e^2-bi$ and $\det M$ are zero; hence, $di-he$ is also zero. Take 
$\{1,2\}$ to be $\{a_1,\dots,a_{r-1}\}$ and columns $\{1,3,4,5\}$
to be $\{b_1,\dots,b_{r+1}\}$ to see that
$$\det \bmatrix b&d\\e&h\endbmatrix 
\det \bmatrix b&e\\d&h\endbmatrix+ b\det M -\det \bmatrix b&d\\d&g\endbmatrix 
\det \bmatrix b&e\\e&i\endbmatrix=0;$$hence $bh-de=0$.
At any rate, the matrix
$$\bmatrix 
b&d&e\\e&h&i\endbmatrix$$ has rank at most one; hence
\begin{align*}x\phi_3={}&bx^*y^*+dy^{*(2)}+ey^*z^*\text{ and}\\
z\phi_3={}&ex^*y^*+hy^{*(2)}+iy^*z^*\end{align*}
are linearly dependent. This  contradicts the hypothesis that $\ell\phi_3$ is nonzero for all nonzero $\ell$ in $U$. 
\end{proof}

\begin{lemma}
\label{8.3} If $\phi_3$ is given by 
{\rm (\ref{8.1.2})},
with $e=f=0$, then
\begin{enumerate}[\rm(a)]
\item\label{8.3.a}$\Gamma_{\phi_3}(y^{(3)}\t \omega_U)=-d^2i\omega_{U^*}$,
\item\label{8.3.b}$\Gamma_{\phi_3}(x^{(2)}z\t \omega_U)=adj\omega_{U^*}$,
\item\label{8.3.c}$\Gamma_{\phi_3}(xy^{(2)}\t \omega_U)=a(gi-h^2)\omega_{U^*}$,
\item\label{8.3.d}$\Gamma_{\phi_3}(xz^{(2)}\t \omega_U)=a(hj-i^2)\omega_{U^*}$, and 
\item\label{8.3.e}$\Gamma_{\phi_3}(xyz\t \omega_U)=a(gj-hi)\omega_{U^*}$.
\end{enumerate}
\end{lemma}
\begin{proof}
(\ref{8.3.a}) We compute 
\begingroup\allowdisplaybreaks\begin{align*}
\Gamma_{\phi_3}(y^{(3)}\t \omega_U)={}&xy\phi_3\w w^2\phi_3\w yz\phi_3\\
{}={}&dy^*\w (dx^*+gy^*+hz^*)\w(hy^*+iz^*)\\
{}={}&-d^2i\omega_{U^*}.\end{align*}\endgroup

\ms\noindent (\ref{8.3.b})
The expansion of $\Gamma_{\phi_3}(x^{(2)}z\t \omega_U)$ has three summands; however two of the summands have a factor of $xz\phi_3=0$; hence
\begingroup\allowdisplaybreaks\begin{align*}\Gamma_{\phi_3}(x^{(2)}z\t \omega_U)={}&
x^2 \phi_3\w xy\phi_3\w z^2\phi_3\\
{}={}&
ax^*\w dy^*\w(iy^*+jz^*)\\
{}={}&adj\omega_{U^*}.\end{align*}\endgroup

\ms\noindent (\ref{8.3.c}) 
The expansion of $\Gamma_{\phi_3}(xy^{(2)}\t \omega_U)$ has three summands; however two of the summands have a factor of $xz\phi_3=0$ or $xy\phi_3\w xy\phi_3=0$; hence 
\begingroup\allowdisplaybreaks\begin{align*}\Gamma_{\phi_3}(xy^{(2)}\t \omega_U)={}&x^2\phi_3\w y^2\phi_3\w yz\phi_3\\
{}={}&ax^*\w (dx^*+gy^*+hz^*)\w (hy^*+iz^*)\\
{}={}&a(gi-h^2)\omega_{U^*}.
\end{align*}\endgroup

\ms\noindent (\ref{8.3.d})
The expansion of $\Gamma_{\phi_3}(xz^{(2)}\t \omega_U)$ has three summands; however two of the summands have a factor of $xz\phi_3=0$; hence 
\begingroup\allowdisplaybreaks\begin{align*}\Gamma_{\phi_3}(xz^{(2)}\t \omega_U)={}&x^2\phi_3\w yz\phi_3\w z^2\phi_3\\
{}={}&ax^*\w (hy^*+iz^*)\w(iy^*+jz^*)\\
{}={}&a(hj-i^2)\omega_{U^*}.
\end{align*}\endgroup

\ms\noindent (\ref{8.3.e})
The expansion of $\Gamma_{\phi_3}(xyz\t \omega_U)$ has six summands; however five of the summands have a factor of $xz\phi_3=0$, $yz\phi_3\w yz\phi_3=0$, or $xy\phi_3\w xy\phi_3=0$; hence 
\begingroup\allowdisplaybreaks\begin{align*}\Gamma_{\phi_3}(xyz\t \omega_U)={}&x^2\phi_3\w y^2\phi_3\w z^2\phi_3\\
{}={}&ax^*\w (dx^*+gy^*+hz^*)\w(iy^*+jz^*)\\
{}={}&a(gj-hi)\omega_{U^*}.\end{align*}
\endgroup
\end{proof}

\begin{lemma}
\label{8.4}
 If $\phi_3$ has the form of {\rm (\ref{8.2.2})},
then $$\Gamma_{\phi_3}(x^{(3)}\t \omega_{U})=-b^2f \omega_{U^*}\quad\text{and}\quad \Gamma_{\phi_3}(x^{(2)}z\t \omega_{U})=-b^2j \omega_{U^*}.$$\end{lemma}

\begin{proof} We compute
\begingroup\allowdisplaybreaks
\begin{align*}\Gamma_{\phi_3}(x^{(3)}\t \omega_{U})={}&x^2\phi_3\w xy\phi_3\w xz\phi_3\\
{}={}&by^*\w (bx^*+dy^*+ez^*)\w (ey^*+fz^*)\\
{}={}&-b^2f\omega_{U^*}.
\end{align*}\endgroup The expansion of $\Gamma_{\phi_3}(x^{(2)}z\t \omega_{U})$
has three summands; however one of the summands has a factor of $xz\phi_3\w xz\phi_3=0$. Thus,   \begingroup\allowdisplaybreaks  
\begin{align*}\Gamma_{\phi_3}(x^{(2)}z\t \omega_{U})={}&
x^2\phi_3\w yz\phi_3\w xz\phi_3+x^2\phi_3\w xy\phi_3\w z^2\phi_3\\
{}={}&\begin{cases}\phantom{+}by^*\w (ex^*+hy^*+iz^*)\w(ey^*+fz^*)\\+by^*\w (bx^*+dy^*+ez^*)\w(fx^*+iy^*+jz^*)\end{cases}\\
{}={}&-bef\omega_{U^*}-b(bj-ef)\omega_{U^*}\\
{}={}&-b^2j\omega_{U^*}.\end{align*}\endgroup
\vskip-24pt\end{proof}

\begin{lemma}
\label{8.5}
 If $\phi_3$ has the form of {\rm (\ref{8.2.2})}, with $f=0$, 
then$$\Gamma_{\phi_3}(x^{(2)}y\t \omega_U)=b(e^2-bi)\omega_{U^*}\quad\text{and}\quad
\Gamma_{\phi_3}(y^{(3)}\t \omega_U)=(\det M)\, \omega_{U^*},$$
where $M$ is the matrix of {\rm(\ref{M2})}.
\end{lemma}

\begin{proof} The expansion of
 $\Gamma_{\phi_3}(x^{(2)}y\t \omega_U)$ has three summands. Two of the summands have a factor of $xy\phi_3\w xy\phi_3=0$ or $x^2\phi_3\w xz\phi_3\in \bw^2\kk y^*=0$. Thus,
\begin{align*}\Gamma_{\phi_3}(x^{(2)}y\t \omega_U)={}
&x^2\phi_3\w xy\phi_3\w yz\phi_3\\
{}={}&by^*\w(bx^*+dy^*+ez^*)\w (ex^*+hy^*+iz^*)\\
{}={}&b(e^2-bi)\omega_{U^*}.\end{align*}We compute
\begin{align*}\Gamma_{\phi_3}(y^{(3)}\t \omega_U)={}&
xy\phi_3\w y^2\phi_3\w yz\phi_3\\
{}={}&
(bx^*+dy^*+ez^*)\w(dx^*+gy^*+hz^*)\w (ex^*+hy^*+iz^*)\\ 
{}={}&
\left|\begin{matrix} b&d&e\\d&g&h\\e&h&i\end{matrix}\right|\omega_{U^*}.\end{align*}
\vskip-24pt\end{proof}

\end{document}